\documentclass[12pt,a4paper,oneside]{article}
\usepackage{arxiv}

\usepackage[utf8]{inputenc}
\usepackage[numbers]{natbib}
\usepackage{graphicx}
\usepackage{subfigure}
\usepackage{placeins} 

\usepackage{url}

\usepackage{amsmath}
\usepackage{amsfonts}       
\usepackage{enumitem}
\usepackage{tabularx,booktabs} 
\usepackage{makecell}
\usepackage{multirow, makecell}

\usepackage{pdflscape}

\usepackage{wrapfig} 

\usepackage{amsthm}
\newtheorem{theorem}{Theorem}[section]
\newtheorem{corollary}[theorem]{Corollary}
\newtheorem{lemma}[theorem]{Lemma}

\newtheorem{remark}[theorem]{Remark}

\newtheorem{assumption}{Assumption}

\usepackage{ulem}
\usepackage{mathtools}

\usepackage[x11names]{xcolor}

\usepackage{hyperref}

\newcommand{\dx}{\dot{x}}
\newcommand{\ddx}{\ddot{x}}
\newcommand{\gf}{\nabla f}
\newcommand{\Hf}{\nabla^2 f}
\newcommand{\N}{\mathbb{N}}
\newcommand{\R}{\mathbb{R}}
\newcommand{\eps}{\varepsilon}

\newcommand{\norm}[1]{\Vert #1\Vert}

\newcommand*\diff{\mathop{}\!\mathrm{d}}

\newcommand{\xlm}{x_{LM}}
\newcommand{\dxlm}{\dot{x}_{LM}}

\title{Continuous Newton-like Methods featuring Inertia and Variable Mass}

\author{\textbf{Camille Castera}$^\ast$\\
	Department of Mathematics\\
	University of Tübingen\\
	Germany
	\and
	\textbf{Hedy Attouch}\\
	IMAG \\
	Universit\'e Montpellier\\
	CNRS, France
	\and
	\textbf{Jalal Fadili}\\
	ENSICAEN
	\\ Normandie Universit\'e
	\\CNRS, GREYC, France
	\vspace{0.2cm}
	\and
	\textbf{Peter Ochs}\\
	Departments of Computer Science and Mathematics\\
	Saarland University\\
	Germany
}

\makeatletter

\begin{document}

	\maketitle

	\vspace{-1cm}
	\begin{center}
	\textit{Dedicated to the memory of Hedy Attouch, outstanding mathematician and beloved collaborator.}
	\end{center}
	\vspace{0.3cm}

	\begin{abstract}
		We introduce a new dynamical system, at the interface between second-order dynamics with inertia and Newton's method.
		This system extends the class of inertial Newton-like dynamics by featuring a time-dependent parameter in front of the acceleration, called \emph{variable mass}.
		For strongly convex optimization, we provide guarantees on how the Newtonian and inertial behaviors of the system can be non-asymptotically controlled by means of this variable mass.
		A connection with the Levenberg--Marquardt (or regularized Newton's) method is also made.
		We then show the effect of the variable mass on the asymptotic rate of convergence of the dynamics, and in particular, how it can turn the latter into an accelerated Newton method.
		We provide numerical experiments supporting our findings. This work represents a significant step towards designing new algorithms that benefit from the best of both first- and second-order optimization methods.
	\end{abstract}

	\renewcommand*{\thefootnote}{$^\ast$}
	\footnotetext[1]{Corresponding author: \texttt{camille.castera@protonmail.com}\\ \textit{Published in SIAM Journal on Optimization 34(1):251–277}}
	\renewcommand*{\thefootnote}{\arabic{footnote}}
	\setcounter{footnote}{0} 

	\section{Introduction}
	\subsection{Problem Statement}\label{sec::pbstatement}
	A major challenge in modern unconstrained convex optimization consists in building fast algorithms while maintaining low computational cost and memory footprint.
	This plays a central role in many key applications such as large-scale machine learning problems or data processing. The problems we are aiming to study are of the form
	\begin{equation*}
		\min_{x\in\R^n} f(x).
	\end{equation*}
	Large values of $n$ demand for algorithms at the interface of first- and second-order optimization.
	Limited computational capabilities explain why gradient-based (first-order) algorithms remain prominent in practice.
	Unfortunately, they often require many iterations, which is true even for the provably best algorithms for certain classes of optimization problems; for example that of convex and strongly convex functions with Lipschitz continuous gradient \citep{polyak1964some,nesterov1983method,nesterov2003introductory}. On the other hand, algorithms using second-order information (the Hessian of $f$)---with Newton's method as prototype---adapt locally to the geometry of the objective, allowing them to progress much faster towards a solution. However, each iteration comes with high computational and memory costs, which highlights a challenging trade-off.
	It is therefore essential to develop algorithms that take the best of both worlds.
	Several quasi-Newton algorithms partly address this issue, for example BFGS methods \citep{broyden1970convergence,fletcher1970new,goldfarb1970family,shanno1970conditioning,liu1989limited}, yet, in very large-scale applications, first-order algorithms often remain the preferred choice.

	In order to reach a new level of efficiency, deep insights into the mechanism and relations between algorithms are required. To that aim, an insightful approach is to see optimization algorithms as discretization of ordinary differential equations (ODEs): for small-enough step-sizes, iterates can be modeled by a continuous-time trajectory \citep{ljung1977analysis,benaim2005stochastic}.
	Obtaining a fast algorithm following this strategy depends on two ingredients: choosing an ODE for which rapid convergence to a solution can be proved, and discretizing it with an appropriate scheme that preserves the favorable properties of the ODE.

	Both steps are highly challenging, our work focuses on the ODE matter. We study the following second-order dynamical system in a general setting:
	\begin{equation}\tag{VM-DIN-AVD}\label{eq::VMDINAVD}
		\varepsilon(t) \ddot{x}(t) + \alpha(t) \dot x(t) + \beta \Hf(x(t))\dot{x}(t)+\gf(x(t))=0,
		\quad t\geq 0,
	\end{equation}
	where $f\colon \R^n\to\R$ is a smooth function, with gradient $\gf$ and Hessian $\Hf$ defined on $\R^n$ equipped with scalar product $\langle\cdot,\cdot\rangle$, and induced norm $\norm{\cdot}$.
	The system is controlled by two functions $\eps,\alpha\colon\R_+\to\R_+$ (where $\R_+=[0,+\infty[$) and a parameter $\beta> 0$ that define the type of dynamics that drives the trajectory (or solution) $x\colon\R_+\to\R^n$, whose first- and second-order derivatives are denoted $\dx$ and $\ddx$ respectively.
	We call the above dynamics \eqref{eq::VMDINAVD}, which stands for ``Variable Mass Dynamical Inertial Newton-like system with Asymptotically Vanishing Damping'' since it generalizes a broad class of ODEs whose original member is DIN \citep{alvarez2002second}, where $\eps$ and $\alpha$ were constant. DIN was then extended to the case of non-constant \emph{asymptotically vanishing dampings} (AVD) $\alpha$ \citep{attouch2016fast}. In this work we introduce the non-constant parameter $\eps$ called \emph{variable mass} (VM) in front of the acceleration $\ddx$, in the same way that $\alpha$ is called (viscous) \emph{damping} by analogy with classical mechanics.
	A key feature of these ODEs, that positions them at the interface of first- and second-order optimization, is that they possess equivalent forms involving only $\nabla f$ but not $\nabla^2f$, significantly reducing computational costs, hence enabling the design of practical algorithms, see e.g., \cite{castera2021inertial,attouch2020first,chen2019first}.
	The key idea behind this is the relation $\Hf(x(t))\dx(t) = \frac{\diff}{\diff t}\gf(x(t))$, see Section~\ref{sec::prelim} for an equivalent formulation of \eqref{eq::VMDINAVD} exploiting this.

	This paper emphasizes the relation between \eqref{eq::VMDINAVD} and special cases.
	Indeed, taking $\eps=\alpha=0$, one obtains\footnote{\ref{eq::CN} is usually considered with $\beta=1$, we put $\beta$ in the system to ease the discussions below.} the Continuous Newton (CN) method \citep{gavurin1958nonlinear}
	\begin{equation}\tag{CN}\label{eq::CN}
		\beta \Hf(x_N(t))\dot{x}_N(t)+\gf(x_N(t))=0, \quad t\geq 0,
	\end{equation}
	known notably for being invariant to affine transformations and yielding fast vanishing of the gradient (see Section~\ref{sec::Control}).
	In fact, this observation shows that \eqref{eq::VMDINAVD} is a singular perturbation of \eqref{eq::CN}, which also justifies the terminology ``Newton-like''  in DIN.
	When $\alpha\neq 0$ but $\eps=0$, we recover the Levenberg--Marquardt (LM) method,
	\begin{equation}\tag{LM}\label{eq::LM}
		\alpha(t) \dot x_{LM}(t) + \beta \Hf(x_{LM}(t))\dot{x}_{LM}(t)+\gf(x_{LM}(t))=0, \quad t\geq 0,
	\end{equation}
	also known as regularized Newton method since it stabilizes \eqref{eq::CN}.
	In the rest of the paper, the solutions of \eqref{eq::CN} and \eqref{eq::LM} will always be denoted by $x_N$ and $x_{LM}$ respectively. Since the introduction of DIN, it is known (see \cite{alvarez2002second}) that for $\alpha=0$, $\beta=1$, and $\eps$ constant and small, \eqref{eq::VMDINAVD} is a ``perturbed'' Newton method since the distance between the solutions of \eqref{eq::VMDINAVD} and \eqref{eq::CN}  is at most proportional to $\sqrt{\eps}$ at all times.
	Yet, despite the benefits of this class of ODEs, such as stabilization properties \citep{attouch2016fast,attouch2020first}, no improvement\footnote{DIN-like systems were thought to yield faster vanishing of the gradient compared to inertial first-order dynamics, until recently \citep{attouch2022ravine}.} of the rate of convergence (in values) has been shown compared to inertial first-order dynamics \citep{polyak1964some,su2014differential}. This raises the question:
	\begin{center}
		``\emph{are these ODEs really of Newton type?}'',
	\end{center}
	which is crucial in view of designing faster algorithms from them.

	\begin{figure}[t]
		\centering
		\begin{minipage}{0.53\linewidth}
			\centering
			\hspace{0.05\linewidth}\includegraphics[width=.8\linewidth]{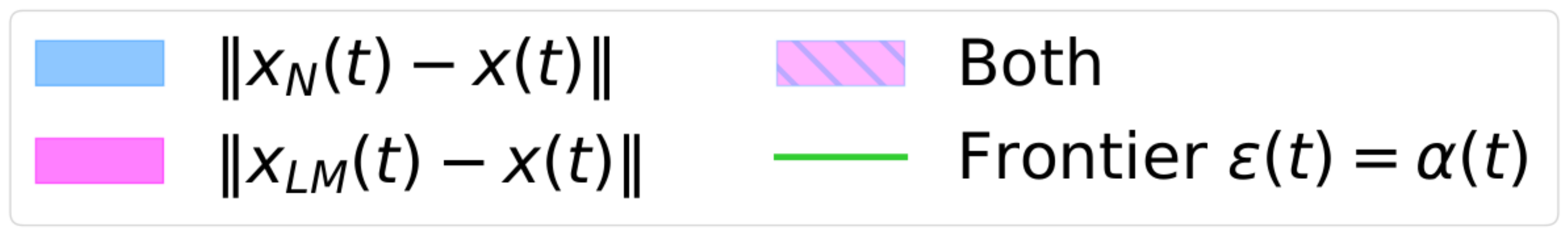}
			\includegraphics[width=\linewidth]{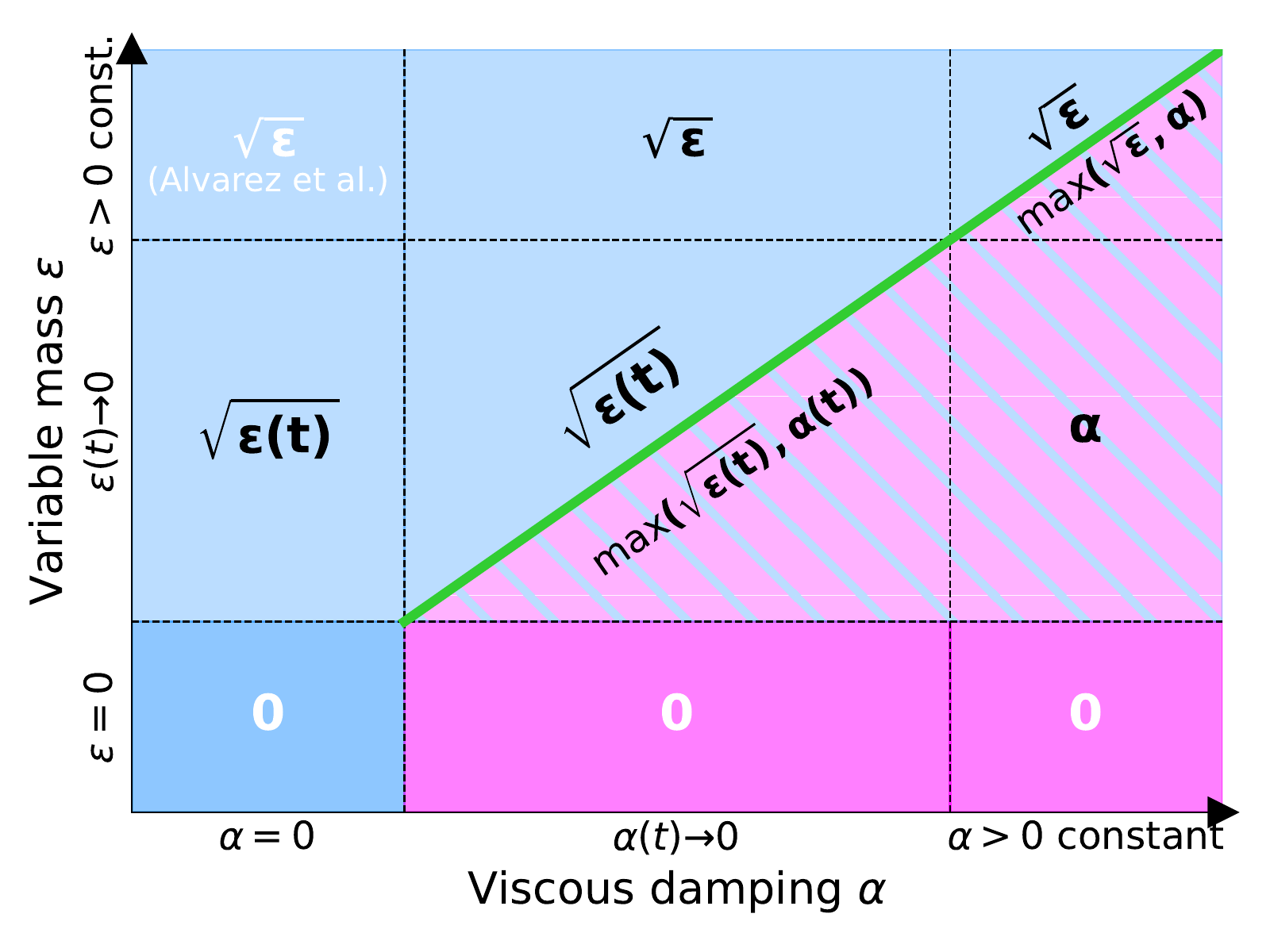}
		\end{minipage}
		\begin{minipage}{0.46\linewidth}
			\centering
			\includegraphics[width=\linewidth]{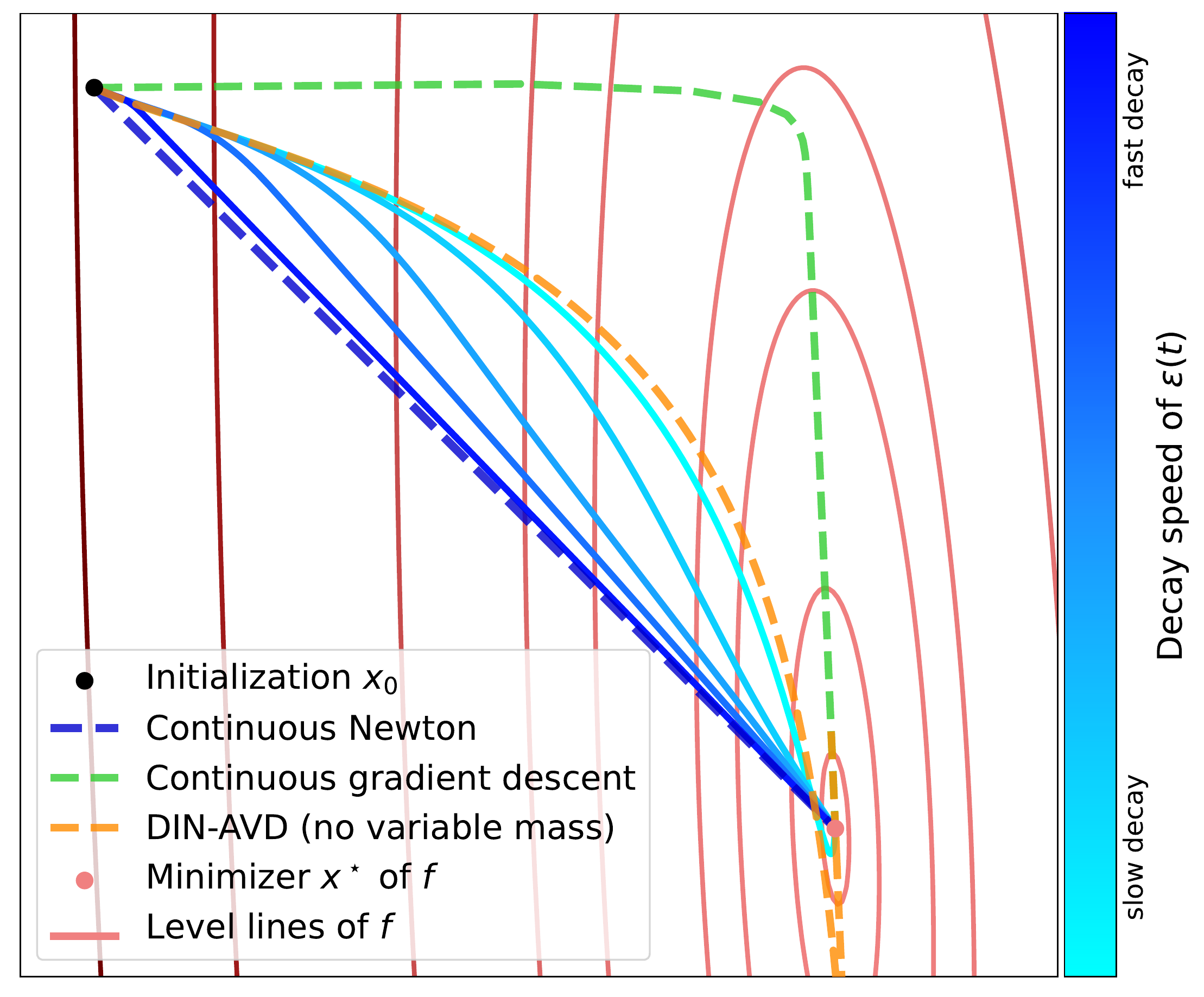}
		\end{minipage}
		\caption{Left: phase diagram on distances from \eqref{eq::VMDINAVD} to \eqref{eq::CN} and \eqref{eq::LM} (see Section~\ref{sec::Control}). The color of each patch indicates which distance is considered, and the scaling of a corresponding upper-bound on this distance is written (in white for prior work, in black for our contributions).
			The green line separates the cases $\eps\geq \alpha$ (above) and $\eps\leq \alpha$ (below).
			Right: 2D illustration of the trajectories of \eqref{eq::VMDINAVD} for several choices of $\eps$ on a quadratic function. Fast-vanishing $\eps(t)$ (dark-blue solid curves) bring solutions of \eqref{eq::VMDINAVD} close to that of \eqref{eq::CN}, making them, more robust to bad conditioning compared to first-order dynamics (e.g., gradient descent).
			\label{fig::phasediagandquad}}
	\end{figure}

	\subsection{Main Contributions} We show that the answer to this question is partially positive, and closely related to the choices of $\eps$ and $\alpha$.
	We provide general results on the role played by these two control parameters and how they can be chosen to control \eqref{eq::VMDINAVD}, and make it close to \eqref{eq::CN} \emph{for all time}, as illustrated on the right-hand side of Figure~\ref{fig::phasediagandquad}, but also to obtain fast convergence. This represents a first step towards building new fast practical algorithms. Our main contributions are the following:
	\\
	--\ We provide a first-order equivalent formulation of \eqref{eq::VMDINAVD}, and show the existence and uniqueness of the solutions of \eqref{eq::VMDINAVD} under mild assumptions.
	\\
	--\ We generalize the perturbed Newtonian property discussed above to non-constant and possibly vanishing variable masses $\eps$, and ``not too large'' positive dampings $\alpha$, and derive bounds that (formally) take the form $\norm{x(t)-x_N(t)}=O(\sqrt{\eps(t)})$. We then extend these results to larger $\alpha$ and make the connection between \eqref{eq::VMDINAVD} and \eqref{eq::LM}. This contribution is summarized in the phase diagram of Figure~\ref{fig::phasediagandquad}.
	\\
	--\ Using quadratic functions as a model for strongly convex functions, we shed light on techniques to efficiently approximate solutions of \eqref{eq::VMDINAVD}.
	We then show how $\eps$ and $\alpha$ affect the speed of convergence. Depending on their setting, the solutions of \eqref{eq::VMDINAVD} may either converge as fast as that of \eqref{eq::CN}, \emph{faster}, or rather have a \eqref{eq::LM} nature, as summarized in Table~\ref{tabl::speedsummary}.
	\\
	--\ We provide numerical experiments supporting our theoretical findings.

	\begin{table}[t]
		\small
		\centering\setlength\extrarowheight{1.5pt}
		\caption{Informal summary of Section~\ref{sec::Quads}. Comparison of \eqref{eq::VMDINAVD} with other dynamics\label{tabl::speedsummary}}
		\begin{tabular}{@{\extracolsep{2pt}}*{2}{c@{\enspace}c@{\enspace}c}}
			\toprule
			\multicolumn{2}{c}{Parameters of \eqref{eq::VMDINAVD}} &\quad  & \multicolumn{2}{c}{Speed of convergence}\\
			\cmidrule{1-2} \cmidrule{4-5}
			{Dominant parameter}& {Integrability in $+\infty$} & & w.r.t.\ \eqref{eq::CN} & w.r.t.\ \eqref{eq::LM} \\
			\midrule
			\multirow{2}{*}{variable mass $\eps$}
			& yes & & as fast & as fast   \\
			& no & & \color{Green3}faster & \color{Green3}faster \\
			\multirow{2}{*}{viscous damping $\alpha$}
			& yes & & as fast & \multirowcell{2}{\color{Seashell4}only depends\\ \color{Seashell4}on $\eps$}  \\
			& no & & \color{IndianRed1}slower & \\
			\bottomrule
		\end{tabular}
	\end{table}

	\subsection{Related work}
	The importance and challenges of finding systems that approximate and preserve the benefits of \eqref{eq::CN} were highlighted by \cite{alvarez1998dynamical}.
	The system \eqref{eq::VMDINAVD} belongs to the class of inertial systems with viscous and geometric (``Hessian-driven'') dampings, initially introduced with constant $\eps = 1$ and constant $\alpha$ in \cite{alvarez2002second} and called DIN (for Dynamical Inertial Newton-like system). Except in a few cases \citep{castera2021inertial,castera2021saddle}, most of the follow-up work then considered extensions of DIN with non-constant AVD $\alpha$, with in particular the DIN-AVD system with $\alpha(t)=\alpha_0/t$ as introduced in \cite{attouch2016fast}. The reason for this popular choice for $\alpha$ is its link with Nesterov's method \citep{su2014differential}. Non-constant choices for $\beta$ have been considered \citep{attouch2020first,alecsa2021extension,lin2022control,attouch2020fast}. We keep it constant, and rather focus on non-constant $\eps$, unlike prior work that used constant $\eps=1$.
	The mass $\eps$ was only considered in the original work \citep{alvarez2002second}, but only for fixed $\eps$, $\beta=1$ and constant $\alpha=0$.
	\ref{eq::VMDINAVD} is however related to the IGS system considered in \cite{attouch2020fast} as it is actually equivalent to the latter after dividing both sides of \eqref{eq::VMDINAVD} by $\eps(t)$. Our approach---which consists in studying the connections with other second-order dynamics as $\eps$ vanishes asymptotically---is however different from the one followed in \citep{attouch2020fast}, which is of independent interest.
	The literature on DIN is rich, let us mention further connections with Nesterov's method \citep{shi2021understanding,alecsa2021extension}, extensions with Tikhonov regularization \citep{boct2021tikhonov} and closed-loop damping \citep{attouch2022fast,lin2022control}. The non-smooth and possibly non-convex cases have been considered in  \cite{attouch2020newton,attouch2021continuous,castera2021inertial}. Finally, avoidance of strict saddle points in smooth non-convex optimization has been shown in \cite{castera2021saddle}.

	The influence of the damping $\alpha$ on the \eqref{eq::LM} dynamics has been studied in \cite{attouch2011continuous,attouch2013global}. Interestingly, the conditions enforced on $\alpha$ in these papers (formally a sub-exponential decay) are very similar to those we make on $\eps$ and $\alpha$ for \eqref{eq::VMDINAVD} (see Assumptions~\ref{ass:largeepsilon} and~\ref{ass:Onlysubexp}). The \eqref{eq::LM} dynamics is also related to the system in \citep{attouch2015dynamic}.

	Regarding the second part of our analysis, which deals with the case where $f$ is quadratic, a recent work \citep{attouch2020first} provided closed-form solutions to \eqref{eq::VMDINAVD} for $\eps\equiv 1$ and special choices of $\alpha$. Our work rather deals with approximate solutions which allows considering a wide class of functions $\eps$ and $\alpha$. We rely on the Liouville--Green (LG) method~\citep{Liouville1836,green1838} presented in Section~\ref{sec::Quads}. Generalizations of LG are also often referred to as WKB methods~\citep{wentzel1926verallgemeinerung,kramers1926wellenmechanik,brillouin1926remarques} and seem to be mostly used in physics so far. To the best of the authors' knowledge, the current work seem to be one of the first to use the LG method in optimization, and the first for DIN-like systems.

	\subsection{Organization}
	The paper is organized as follows. We discuss the existence of solutions in Section~\ref{sec::prelim}. Our main results, from a non-asymptotic control perspective, are then presented in Section~\ref{sec::Control}.
	An analysis of the role played asymptotically by $\eps$ and $\alpha$ is then carried out on quadratic functions in Section~\ref{sec::Quads}. Finally, numerical experiments are presented in Section~\ref{sec::exp}, and some conclusions are then drawn.

	\section{Existence and Uniqueness of Solutions}\label{sec::prelim}
	We always assume the following on the system \eqref{eq::VMDINAVD} and the functions $f,\eps,\alpha$ defined in the introduction.
		\begin{assumption}\label{ass:general}
			\textbullet\ $f: \R^n \to \R$ is convex, twice continuously differentiable, coercive, and strongly convex on bounded subsets of $\R^n$.
			We fix $x_0$, $\dot{x}_0\in\R^n$, such that, unless stated otherwise, \eqref{eq::VMDINAVD} has initial condition $(x(0),\dx(0))=(x_0,\dx_0)$, and \eqref{eq::CN} and \eqref{eq::LM} have initial conditions $x_{N}(0)=x_{LM}(0)=x_0$.
\\
			\textbullet\ $\alpha: \R_+ \to \R$ is non-increasing, non-negative and differentiable. $\eps: \R_+ \to \R$ is non-increasing, positive, and twice differentiable with bounded second derivative.
			We fix initial values: $\eps(0)=\eps_0>0$, $\eps'(0)=\eps'_0\leq0$ and $\alpha(0)=\alpha_0\geq 0$.
		\end{assumption}

	\begin{theorem}\label{thm::existence}
			Under Assumption~\ref{ass:general} there exists a unique global solution $x:\R_+\to \R^n$ to \eqref{eq::VMDINAVD}.
	\end{theorem}
	The proof relies on the Cauchy--Lipschitz Theorem, we sketch the main elements. We reformulate \eqref{eq::VMDINAVD} into a first-order (in time) system by introducing an auxiliary variable $y:\R_{+}\to\R^n$. Notably, our reformulation does not involve $\nabla^2 f$, in the same fashion as \cite{alvarez2002second,attouch2016fast}.
	For all $t$, defining $\nu(t) = \alpha(t)-\eps'(t) - \frac 1 \beta \eps(t)$, we show in Appendix~\ref{app:FOS} that \eqref{eq::VMDINAVD} is equivalent to
	\begin{equation}\label{eq::gVMDINAVD}\tag{gVM-DIN-AVD}
		\begin{cases}
			\eps(t)\dx(t) + \beta\gf(x(t)) + \nu(t)x(t) + y(t)&= 0
			\\\dot y(t)+ \nu'(t)x(t) + \frac{\nu(t)}{\beta} x(t) + \frac 1 \beta y(t) &= 0
		\end{cases}
	\end{equation}
	with initial conditions $(x(0),y(0)) = \left(x_0,-\eps_0\dx_0 - \beta\gf(x_0) - (\alpha_0-\eps'_0-\frac{1}{\beta}\eps_0)x_0\right)$.
	One can notice that in the special case where $\eps$ is taken constant and equal to $1$ (that is when \eqref{eq::VMDINAVD} is simply the DIN-AVD system \citep{attouch2016fast}), we recover the same first-order formulation as that in \cite{attouch2016fast}.
	For all $t\geq0$ and $(u,v)\in\R^n\times\R^n$, define the mapping
	$ G\left(t,(u,v)\right) = \begin{pmatrix}
		\frac{1}{\eps(t)}(-\beta\gf(u)-\nu(t)u-v)
		\\ -\nu'(t)u - \frac{\nu(t)}{\beta}u - \frac{1}{\beta}v
	\end{pmatrix}, $
	so that \eqref{eq::gVMDINAVD} rewrites $(\dx(t),\dot y(t)) = G\left(t,(x(t),y(t))\right)$ for all $t\geq 0$. Since $f$ is twice continuously differentiable, one can see that $G$ is continuously differentiable w.r.t.\ its second argument $(u,v)$. Consequently $G$ is locally Lipschitz continuous w.r.t.\ $(u,v)$ and by the Cauchy--Lipschitz Theorem, for each initial condition, there exists a unique local solution to \eqref{eq::gVMDINAVD} and thus to \eqref{eq::VMDINAVD}. We then show that this solution is actually global (in Appendix~\ref{app::localisglobal}) by proving the boundedness of $(x,y)$.

	We omit the existence and uniqueness of the solutions of \eqref{eq::CN} and \eqref{eq::LM} since these are standard results, obtained with similar arguments. We conclude with the following important remark.
		\begin{remark}\label{rem::nonsmooth}
			Thanks to the first-order reformulation \eqref{eq::gVMDINAVD}, Theorem~\ref{thm::existence} does not require the Lipschitz continuity of $\nabla^2 f$ and not even that of $\nabla f$ to obtain global existence and uniqueness. In contrast, these smoothness assumptions would have been required if the standard first-order formulation of \eqref{eq::VMDINAVD} in phase-space (position-velocity) was used. In fact, \eqref{eq::gVMDINAVD} allows for a natural extension of \eqref{eq::VMDINAVD} to non-smooth convex functions as we detail in Appendix~\ref{app:nonsmooth}, and which is not the case for \eqref{eq::CN} and \eqref{eq::LM}. This allows defining an inertial Newton-like dynamics in the non-smooth setting which is important for many applications; e.g.\ \cite{castera2021inertial}.
		\end{remark}


	\section{Non-asymptotic Control of \texorpdfstring{\eqref{eq::VMDINAVD}}{VM-DIN-AVD}}\label{sec::Control}
	The purpose of this section is to understand how close $x$ might be to $x_N$ and $x_{LM}$, as a function of $\alpha$ and $\eps$.
	Since $f$ is coercive and strongly convex on bounded sets, it has a unique minimizer $x^\star\in \R^n$. Consequently, any two trajectories that converge to $x^\star$ will eventually be arbitrarily close to each other. Thus, asymptotic results of the form $\Vert x(t)-x_N(t)\Vert \xrightarrow[t\to+\infty]{} 0$ are not precise enough to claim, for example, that $x$ has a ``Newtonian behavior''.
	Instead, we will derive upper bounds on the distance between trajectories that hold \emph{for all time} $t\geq 0$, and which typically depend on $\eps$ and/or $\alpha$. We first present the case where $\alpha$ is small relative to $\eps$ and then generalize.

	\subsection{Comparison with \texorpdfstring{\eqref{eq::CN}}{CN} under Moderate Viscous Dampings}
	When the damping $\alpha$ remains moderate w.r.t.\ the variable mass $\eps$, one expects the solutions of \eqref{eq::VMDINAVD} to be close to that of \eqref{eq::CN}. We make the following assumptions.
	\begin{assumption}\label{ass:largeepsilon}
		Assumption~\ref{ass:general} holds and there exists $c_1,c_2\geq 0$ such that for all $t\geq 0$, $\vert\eps'(t)\vert\leq c_1 \eps(t)$ and $\alpha(t) \leq c_2\eps(t)$.
	\end{assumption}
	The assumption states that $\alpha$ must decrease at least as fast as $\eps$ (up to a constant).\footnote{Assumption~\ref{ass:largeepsilon} can hold only after some $t_0\geq 0$, we take $t_0=0$ for the sake of simplicity.}
	The reason for assuming $\vert\eps'(t)\vert \leq  c_1\eps(t)$ is technical and will appear more clearly in the proofs below. It formally means that $\eps$ can decrease at most exponentially fast.\footnote{This is a consequence of Gronwall's lemma, see e.g., \citep{emmrich1999discrete}.} This is a relatively mild assumption that holds, for example, for any polynomial decay $\eps_0/(t+1)^a$, $a\in\N$.
	We start with the main result of this section.
	\begin{theorem}\label{thm::mainEpsLarge}
		Let $x_N$ be the solution of \eqref{eq::CN}, and let $c_1,c_2\geq 0$. There exist $C_0,C_1,C_2\geq 0$, depending on $c_1$, $c_2$, such that for all $(\eps,\alpha)$ for which Assumption~\ref{ass:largeepsilon} holds with constants $c_1$ and $c_2$, the corresponding solution $x$ of \eqref{eq::VMDINAVD} is such that for all $t\geq 0$,
		\begin{equation}\label{eq::ResLargeEps}
			\norm{x(t)-x_N(t)}\leq C_0 e^{-\frac{t}{\beta}}\eps_0\norm{\dx_0} + C_1\sqrt{\eps(t)} +C_2 \int_{s=0}^{t} e^{\frac{1}{\beta}(s-t)} \sqrt{\eps(s)}\diff s.
		\end{equation}
	\end{theorem}
	This extends a previous result from \cite[Proposition~3.1]{alvarez2002second} which states a similar bound for constant $\eps$, $\alpha\equiv 0$ and $\beta=1$. Theorem~\ref{thm::mainEpsLarge} corresponds to the blue parts in the phase diagram of Figure~\ref{fig::phasediagandquad} (see also Corollary~\ref{cor::simpler} below).
	\begin{remark}
		The strength of the result comes from the fact that $C_0,C_1,C_2$ do not depend on $\eps$ and $\alpha$, and that the result is \emph{non-asymptotic}. This allows in particular for choosing $(\eps,\alpha)$ to control the distance from $x$ to $x_N$, for all time $t\geq 0$.
	\end{remark}
	\begin{remark}
		Under Assumption~\ref{ass:largeepsilon}, the dynamics \eqref{eq::VMDINAVD} is dominated by the variable mass. The damping $\alpha$ does not appear in Theorem~\ref{thm::mainEpsLarge}.
	\end{remark}
	The above theorem and remarks emphasize the ``Newtonian nature'' of \ref{eq::VMDINAVD}. We present two lemmas before proving Theorem~\ref{thm::mainEpsLarge}, and then state a simpler bound than \eqref{eq::ResLargeEps}, see Corollary~\ref{cor::simpler}.
	\begin{lemma}\label{lem::LyapLargeEps}
		Let $(\eps,\alpha)$, and let $x$ be the corresponding solution of \ref{eq::VMDINAVD}. For all $t\geq 0$, define the function,
		$U(t) = \frac{\eps(t)}{2} \norm{\dot{x}(t)}^2 + f(x(t))-f(x^\star).$
		Then $U$ is differentiable and for all $t>0$,
		\begin{equation*}
			\frac{\diff U}{\diff t} (t) = \frac{\eps'(t)}{2} \norm{\dot{x}(t)}^2 - \alpha(t) \norm{\dot{x}(t)}^2 - \beta\langle \Hf(x(t))\dot{x}(t), \dot{x}(t)\rangle\leq 0.
		\end{equation*}
		Therefore, in particular, $U$ is non-increasing.
	\end{lemma}

	\begin{proof}
		Let $t\geq 0$, since $x$ is twice differentiable, $U$ is differentiable and,
		\begin{equation*}
			\frac{\diff U}{\diff t} (t) = \frac{\eps'(t)}{2} \norm{\dot{x}(t)}^2 + \eps(t)\langle\dx(t),\ddx(t)\rangle + \langle \dot{x}(t),\gf(x(t))\rangle.
		\end{equation*}
		We use the fact that $x$ is solution of \eqref{eq::VMDINAVD}, to substitute $\eps(t)\ddx(t)$ by its expression,
		$
		\frac{\diff U}{\diff t} (t) = \frac{\eps'(t)}{2} \norm{\dot{x}(t)}^2 - \alpha(t) \norm{\dot{x}(t)}^2 - \beta\langle \Hf(x(t))\dot{x}(t), \dot{x}(t)\rangle.
		$
		By assumption $\eps$ is non-increasing so for all $t>0$, $\eps'(t)\leq 0$. Furthermore $f$ is convex so $\langle \Hf(x(t))\dot{x}(t), \dot{x}(t)\rangle\geq 0$. Hence $U$ is non-increasing.
	\end{proof}
	We then state the following bound.
	\begin{lemma}\label{lem::boundeps}
		For any $(\eps,\alpha)$, the corresponding solution  $x$ of \eqref{eq::VMDINAVD} is such that for all $t\geq 0$,
		\begin{equation*}
				\eps(t)\norm{\dx(t)}\leq \sqrt{2U(0)}\sqrt{\eps(t)}.
		\end{equation*}
	\end{lemma}
	From Lemma~\ref{lem::LyapLargeEps}, $U$ is non-increasing so $\forall t\geq 0$, $U(t)\leq U(0)$. Then in particular
		$
		\frac{\eps(t)}{2} \norm{\dot{x}(t)}^2 \leq U(0)
		$
		and the proof follows by multiplying both sides by $\eps(t)$.

	\begin{proof}[Proof of Theorem~\ref{thm::mainEpsLarge}]
		Let $(\eps,\alpha)$ as defined in Sections~\ref{sec::pbstatement} and~\ref{sec::prelim}, and let $x$ be the corresponding solution of \eqref{eq::VMDINAVD}. Then, according to Lemma~\ref{lem::LyapLargeEps}, for all $t\geq 0$,  $U(t)\leq U(0)$, so  in particular
		$f(x(t)) \leq f(x_0) + \frac{\eps_0}{2}\Vert\dx_0\Vert^2.$
		Denoting $M_0= f(x_0) + \frac{\eps_0}{2}\Vert\dx_0\Vert^2$, the set $\mathsf{K}_0=\left\{y\in\R^n\mid f(y)\leq M_0\right\}$ is bounded, since $f$ is coercive ($\lim_{\norm{y}\to+\infty} f(y)=+\infty$). So for all $t\geq 0$, $x(t)\in\mathsf{K}_0$. Since $M_0$ (and hence $\mathsf{K}_0$) depends only on $\eps_0$, $x_0$ and $\dx_0$, we have proved that for any choice $(\eps,\alpha)$, the corresponding solution $x$ of \eqref{eq::VMDINAVD} is inside $\mathsf{K}_0$ at all times.
		Let $x_N$ be the solution of \eqref{eq::CN}. One can similarly see that for all $t\geq 0$, $f(x_N(t))\leq f(x_N(0))=f(x_0)\leq M_0$. So we also have $x_N(t)\in \mathsf{K}_0$ for all $t\geq 0$.

		Now, fix $c_1,c_2>0$, and let $(\eps,\alpha)$ such that Assumption~\ref{ass:largeepsilon} is satisfied with constants $c_1,c_2$. Let $x$ be the corresponding solution of \eqref{eq::VMDINAVD}.
		Since $f$ is strongly convex on bounded sets, it is strongly convex on $\mathsf{K}_0$. We denote $\mu_{\mathsf{K}_0}>0$ the strong-convexity parameter of $f$ on $\mathsf{K}_0$. Equivalently, we have that $\nabla f$ is strongly monotone on $\mathsf{K}_0$, that is, $\forall y_1,y_2\in\mathsf{K}_0$,
		\begin{equation}\label{eq::gfMono}
			\langle\gf(y_1)-\gf(y_2),y_1-y_2\rangle\geq \mu_{\mathsf{K}_0} \norm{y_1-y_2}^2.
		\end{equation}
		Let $t\geq 0$, since $x(t)\in\mathsf{K}_0$ and $x_N(t)\in \mathsf{K}_0$, by
		combining \eqref{eq::gfMono} with the Cauchy--Schwarz inequality, we deduce that
		\begin{equation}\label{eq::boundwithgrad}
			\norm{x(t)-x_N(t)}\leq \frac 1 {\mu_{\mathsf{K}_0}}\norm{\gf(x(t))-\gf(x_N(t))}.
		\end{equation}
		Therefore, it is sufficient to bound the difference of gradients in order to bound $\norm{x(t)-x_N(t)}$.
		First, remark that \eqref{eq::CN} can be rewritten as follows: $\frac{\diff}{\diff t} \gf(x_N(t)) + \frac 1 \beta\gf(x_N(t)) = 0$. So we can integrate, for all $t\geq0$,
		\begin{equation}\label{eq::gradN}
			\gf(x_N(t)) = e^{-\frac t \beta}\gf(x_0).
		\end{equation}

		We now turn our attention to $\gf(x(t))$, for which we cannot find a closed-form solution in general. We rewrite \eqref{eq::VMDINAVD} in the following equivalent form
		\begin{equation*}
			\frac{\diff}{\diff t}\left[\eps(t)\dx(t)+\beta\gf(x(t))\right] + \frac{1}{\beta}\eps(t)\dx(t)+\gf(x(t)) = \left(\frac{1}{\beta}\eps(t) + \eps'(t)- \alpha(t)\right) \dx(t).
		\end{equation*}
		Introducing the variable $\omega(t)=\eps(t)\dx(t)+\beta\gf(x(t))$, the latter is thus solution to
		\begin{equation*}
			\dot\omega(t)+ \frac 1 \beta \omega(t) = \left(\frac{1}{\beta}\eps(t) + \eps'(t)- \alpha(t)\right) \dx(t),\quad t\geq 0,
		\end{equation*}
		with $\omega(0) = \eps_0\dot{x}_0+\beta\gf(x_0)$. This is a non-homogeneous first-order ODE in $\omega$, whose solution can be expressed using the integrating factor
		\begin{equation*}
			\omega(t)=e^{-\frac t \beta}(\eps_0\dx_0+\beta\gf(x_0))+ e^{-\frac t \beta}\int_{0}^{t} e^{\frac s \beta} \left(\frac 1 \beta \eps(s) + \eps'(s) - \alpha(s)\right)\dx(s)\diff s.
		\end{equation*}
		We thus have the following expression for $\gf(x)$, for all $t\geq 0$,
		\begin{multline}\label{eq::gradDIN}
			\beta\gf(x(t))=\beta e^{-\frac t \beta}\gf(x_0) + e^{-\frac t \beta }\eps_0\dx_0 -\eps(t)\dx(t) \\+ e^{-\frac t \beta}\int_{0}^{t} e^{ \frac s \beta} \left(\frac 1 \beta \eps(s) + \eps'(s) - \alpha(s)\right)\dx(s)\diff s.
		\end{multline}

		We can now use \eqref{eq::gradN} and \eqref{eq::gradDIN} in \eqref{eq::boundwithgrad} to get
		\begin{multline*}
			\norm{x(t)-x_N(t)}\leq
			\\
			\frac 1 {\beta\mu_{\mathsf{K}_0}}\left\Vert e^{-\frac t \beta }\eps_0\dx_0 -\eps(t)\dx(t) + e^{-\frac t \beta}\int_{0}^{t} e^{ \frac s \beta} \left(\frac 1 \beta \eps(s) + \eps'(s) - \alpha(s)\right)\dx(s)\diff s\right\Vert.
		\end{multline*}
		Using the triangle inequality, we obtain,
		\begin{multline}\label{eq::boundInt2}
			\norm{x(t)-x_N(t)}\leq \frac{\eps_0\Vert\dx_0\Vert}{\beta\mu_{\mathsf{K}_0}} e^{-\frac t \beta } +  \frac{\eps(t)\Vert\dx(t)\Vert}{\beta\mu_{\mathsf{K}_0}}
			\\
			+ \frac{1}{\beta\mu_{\mathsf{K}_0}} \int_{0}^{t} e^{ \frac 1 \beta (s-t)} \left\vert \frac 1 \beta \eps(s) + \eps'(s) - \alpha(s)\right\vert \Vert\dx(s)\Vert\diff s.
		\end{multline}
		The first term in \eqref{eq::boundInt2} corresponds to the first one in \eqref{eq::ResLargeEps} with $C_0= 1/(\beta\mu_{\mathsf{K}_0})$. As for the second-one, by direct application of Lemma~\ref{lem::boundeps}, for all $t\geq 0$, ${\eps(t)\Vert\dx(t)\Vert}\leq \sqrt{2U(0)}\sqrt{\eps(t)}$, so we set $C_1 = \sqrt{2U(0)}/(\beta\mu_{\mathsf{K}_0})$.
		Regarding the last term in \eqref{eq::boundInt2}, using Assumption~\ref{ass:largeepsilon} and again Lemma~\ref{lem::boundeps}, it holds that, for all $s\geq 0$,
		$$
		\left\vert \frac 1 \beta \eps(s) + \eps'(s) - \alpha(s)\right\vert \norm{\dx(s)} \leq
		\left(\frac 1 \beta+c_1+c_2\right) \eps(s)\norm{\dx(s)}\leq (\frac 1 \beta+c_1+c_2)\sqrt{2U(0)\eps(s)}.
		$$
		This proves the theorem with $C_2 = \frac{\sqrt{2U(0)}}{\beta\mu_{\mathsf{K}_0}}\left(\frac 1 \beta+c_1+c_2\right)$.
	\end{proof}

	Let us analyze the bound in Theorem~\ref{thm::mainEpsLarge}. The first term in \eqref{eq::ResLargeEps} decays exponentially fast and can even be zero if the initial speed is $\dx_0=0$, the second-one decays like $\sqrt{\eps(t)}$, however, the rate at which the last term decreases is less obvious. The following corollary gives a less-tight but easier-to-understand estimate.
	\begin{corollary}\label{cor::simpler}
		Let the same assumptions and variables as in Theorem~\ref{thm::mainEpsLarge}. If furthermore $c_1<\frac{2}{\beta}$, then there exists $C_3>0$ such that for all $t\geq 0$,
		\begin{equation*}
			\norm{x(t)-x_N(t)}\leq C_0 e^{-\frac{t}{\beta}}\eps_0\norm{\dx_0} + C_3\sqrt{\eps(t)}.
		\end{equation*}
	\end{corollary}

	\begin{proof}[Proof of Corollary~\ref{cor::simpler}]
		For all $t\geq 0$, define $J(t) = \int_{0}^{t} e^{\frac{s}{\beta}} \sqrt{\eps(s)} \diff s$. An integration by parts yields
		\begin{multline}\label{eq::boundJ}
			J(t)  = \left[\beta e^{\frac{s}{\beta}}\sqrt{\eps(s)}\right]_{s=0}^t - \int_{s=0}^t \beta e^{\frac{s}{\beta}} \frac{\eps'(s)}{2\sqrt{\eps(s)}}\diff s
			\\= \beta e^{\frac{t}{\beta}}\sqrt{\eps(t)} - \beta \eps_0 + \int_{s=0}^t \beta e^{\frac{s}{\beta}} \frac{-\eps'(s)}{2\eps(s)}\sqrt{\eps(s)}\diff s.
		\end{multline}
		By assumption, $0\leq c_1<\frac{2}{\beta}$ such that for all $s>0$, $\vert\eps'(s)\vert \leq c_1\eps(s)$, which in our setting is equivalent to $\frac{-\eps'(s)}{\eps(s)}\leq c_1$.
		So we deduce from \eqref{eq::boundJ} that
		\begin{equation*}
			J(t)  \leq  \beta e^{\frac{t}{\beta}}\sqrt{\eps(t)} + c_1\frac{\beta}{2}\int_{s=0}^t e^{\frac{s}{\beta}} \sqrt{\eps(s)}\diff s = \beta e^{\frac{t}{\beta}}\sqrt{\eps(t)} + c_1\frac{\beta}{2}J(t).
		\end{equation*}
		So, $\left(1-c_1\frac \beta 2\right)J(t)  \leq   \beta e^t\sqrt{\eps(t)}$. By assumption $1-c_1\frac \beta 2>0$, therefore, $J(t)\leq 	\frac{2}{2-c_1\beta} e^{\frac t \beta}\sqrt{\eps(t)}$.
		We use this in \eqref{eq::ResLargeEps} and set $C_3 = C_1 + C_2\frac{2}{2-c_1\beta}$ to get the result.
	\end{proof}

	\begin{remark}\label{rem::extensions}
			Observe that local strong convexity is only required to get \eqref{eq::gfMono} and \eqref{eq::boundwithgrad}. In fact, local strong convexity can be greatly weakened and our claims above can be generalized to a large sub-class of strictly convex functions as we explain in Appendix~\ref{app::strictconvex}. We did not directly consider the strictly convex case to emphasize the main ideas and ease the reading.
			Following Remark~\ref{rem::nonsmooth}, it seems also that a non-smooth extension is possible using regularization techniques. Unlike the strictly convex setting, the non-smooth one would however require further investigations since the vanilla Newton method is not applicable to non-smooth functions.
	\end{remark}

	So far our results only cover the case where $\alpha$ is ``not too large'' w.r.t.\ $\eps$, and do not study \eqref{eq::LM}.
	We now state a more general result that covers these cases.

	\subsection{Generalization to Sub-exponentially Decaying Viscous Dampings}
	This time we do not assume a link between $\eps$ and $\alpha$ but only sub-exponential decays.
	\begin{assumption}\label{ass:Onlysubexp}
		Assumption~\ref{ass:general} holds and there exists $c_1,c_2\geq 0$ such that for all $t\geq 0$, $\vert\eps'(t)\vert\leq c_1 \eps(t)$ and $\vert\alpha'(t)\vert \leq c_2\alpha(t)$.
	\end{assumption}
	We are now in position to state the main result of this section.
	\begin{theorem}\label{thm::mainGenResult}
		Let $x_N$ and $\xlm$ be the solutions of \eqref{eq::CN} and \eqref{eq::LM} respectively, and let $c_1,c_2\geq 0$.
		There exist constants $C,\tilde{C}\geq 0$, depending on $c_1$, $c_2$, $\eps_0$, $\alpha_0$ and the initial conditions, such that for all $\eps$ and $\alpha$ for which Assumption~\ref{ass:Onlysubexp} holds with $c_1$, $c_2$, the corresponding solution $x$ of \eqref{eq::VMDINAVD} is such that for all $t\geq 0$,
		\begin{equation}\label{eq::BoundXN}
			\norm{x(t)-x_N(t)}\leq C\left[ e^{-\frac{t}{\beta}} + \sqrt{\eps(t)} +
			\alpha(t)
			+ \int_{s=0}^{t} e^{\frac{1}{\beta}(s-t)} (\sqrt{\eps(s)}+\alpha(s))\diff s\right],
		\end{equation}
		and,
		\begin{equation}\label{eq::BoundXLM}
			\norm{x(t)-\xlm(t)}\leq \tilde{C}\left[ e^{-\frac{t}{\beta}} + \sqrt{\eps(t)} +
			\alpha(t)
			+ \int_{s=0}^{t} e^{\frac{1}{\beta}(s-t)} (\sqrt{\eps(s)}+\alpha(s))\diff s\right].
		\end{equation}
	\end{theorem}
	The proof is omitted but key elements are presented in Appendix~\ref{supp::GenThm}. Although it follows a similar reasoning as that of Theorem~\ref{thm::mainEpsLarge}, more involved estimates are needed.

	Let us comment on these results. The bound \eqref{eq::BoundXN} generalizes Theorem~\ref{thm::mainEpsLarge}, although the constant involved will, in general, be larger than those in \eqref{eq::ResLargeEps} (see the proof of Theorem~\ref{thm::mainGenResult} in Appendix~\ref{supp::GenThm}).
	Theorem~\ref{thm::mainGenResult} allows for far more flexibility in the choice of $\eps$ and $\alpha$ in order to control $x$ and make it possibly close to $x_N$.
	The bound in \eqref{eq::BoundXLM} is the same as that in  \eqref{eq::BoundXN} (up to a constant), but this time w.r.t.\ $\xlm$, thus connecting \eqref{eq::VMDINAVD} to \eqref{eq::LM}. We make the following two important remarks. First \eqref{eq::BoundXLM} involves  $\alpha$, suggesting that making $x$ close to $\xlm$  requires not only $\eps$ but also $\alpha$ to vanish asymptotically.
	Additionally, Theorem~\ref{thm::mainGenResult} does not state to which of $x_{N}$ and $\xlm$ the solution of \eqref{eq::VMDINAVD} is the closest.
	It remains an open question to know whether one can make \eqref{eq::BoundXLM} independent of $\alpha$, and to state to which trajectory $x$ is the closest. Yet, the numerical experiments in Section~\ref{sec::exp} suggest that neither are possible. Indeed, we observe that for some functions $f$, $x$ is \emph{sometimes} closer to $x_{N}$ than to $\xlm$, even when $\eps(t)\leq \alpha(t)$.

	Nevertheless, Theorem~\ref{thm::mainGenResult} answers the question asked in the introduction: yes, \eqref{eq::VMDINAVD} is really of second-order nature since it can be brought close to the second-order dynamics \eqref{eq::CN} and \eqref{eq::LM}. Doing so, it benefits from the good properties of these methods, such as the robustness to bad conditioning, as previously illustrated on the right of Figure~\ref{fig::phasediagandquad}.
	This concludes the analysis from a control perspective. We will now derive an approximation of the solution $x$ in order to study the impact that $\eps$ and $\alpha$ have on the speed of convergence of $x$ to $x^\star$ compared to the speeds of convergence of $x_N$ and $\xlm$.


	\section{Approximate Solutions and Asymptotic Analysis on Quadratics}\label{sec::Quads}
	In addition to Assumption~\ref{ass:general}, we consider the case where $f$ is a strongly convex quadratic function in order to study the asymptotic behavior of \eqref{eq::VMDINAVD} w.r.t.\ \eqref{eq::CN} and \eqref{eq::LM}. Quadratic functions are the prototypical example of strongly convex functions. In particular, any strongly convex function can be locally approximated by a quadratic one around its minimizer, making the latter a good model for understanding the local behavior of dynamics.
	In this section, $f$ is quadratic: $f(y)=\frac{1}{2}\Vert Ay-b\Vert_2^2$ for all $y\in \R^n$, where $A\in\R^{n\times n}$ is symmetric positive definite and $b\in\R^n$. Without loss of generality, we take $b=0$, so that the unique minimum is $x^\star=0$.
	\subsection{Setting: the Special Case of Quadratic Functions}
	Quadratic functions are particularly interesting in our setting since DIN-like ODEs take a simpler form (as observed in \cite{attouch2020first,shi2021understanding}). Indeed, $\forall y\in\R^n$, $\gf(y) = A^TAy$ and $\Hf(y)=A^T A$. Since $\Hf(y)$ is independent of $y$ we can rewrite \eqref{eq::VMDINAVD} in an eigenspace\footnote{This can be generalized to the case where $A^TA$ is only semi-definite by considering orthogonal projections on an eigenspace spanned by the positive eigenvalues of $A^TA$.} of  $A^TA$. That is, we can study the ODE coordinate-wise by looking at one-dimensional problems of the form
	\begin{equation}\label{eq::QuadeDINAVD}\tag{Q1-VM-DIN-AVD}
		\eps(t)\ddx(t)+(\alpha(t)+\beta\lambda)\dx(t)+\lambda x(t)=0, \quad t\geq 0.
	\end{equation}
	Here (and throughout what follows) $\lambda>0$ denotes any eigenvalue of $A^TA$ and $x\colon\R_+\to\R$ now denotes the corresponding  coordinate (function) of the solution of \eqref{eq::VMDINAVD} in an eigenspace of $A^TA$. The dynamics \eqref{eq::QuadeDINAVD} is a \emph{linear} second-order ODE in $x$ with non-constant coefficients. Similarly, \eqref{eq::LM} can be rewritten coordinate-wise as
	\begin{equation}\label{eq::quadLM}\tag{Q1-LM}
		(\alpha(t)+\beta\lambda)\dx_{LM}(t)+\lambda x_{LM}(t)=0, \quad t\geq 0,
	\end{equation}
	where $\xlm\colon\R_+\to\R$, and \eqref{eq::CN} becomes
	\begin{equation}\label{eq::quadCN}\tag{Q1-CN}
		\beta\dx_N(t)+ x_N(t)=0, \quad t\geq 0,
	\end{equation}
	where again, $x_N\colon\R_+\to\R$ is one-dimensional.
	Observe in particular that \eqref{eq::CN} and \eqref{eq::LM} are now first-order \emph{linear} ODEs, whose solutions have closed forms: $\forall t\geq 0$,
	\begin{equation}\label{eq::closeformLMN}
		x_N(t) = x_0e^{-\frac{t}{\beta}}
		\quad \text{and}\quad
		x_{LM}(t) = x_0\exp\left(-\int_0^t \frac{\lambda}{\alpha(s)+\beta\lambda}\diff s \right).
	\end{equation}
	Since the minimizer is $x^\star=0$, we see that $x_N$ converges exponentially fast to $x^\star$, with a rate independent of $\lambda$  while the rate of $\xlm$ depends on $\lambda$ and how fast $\alpha$ vanishes.

	Unfortunately, except for some special choices of $\eps$ and $\alpha$ (see \cite{attouch2020first}), one cannot solve the second-order linear ODE \eqref{eq::QuadeDINAVD} in closed form in general.
	Additionally, it is hopeless to circumvent the difficulty by finding a closed form for $\gf(x)$, accordingly to what we did in Section~\ref{sec::Control}, since here $\gf(x)=\lambda x$.
	In order to study the speed of convergence of $x$ despite not having access to a closed form, we will approximate it with a controlled error, via a method that we now present.

	\subsection{The Liouville--Green Method}
	In what follows, we rely on the Liouville--Green method~\citep{Liouville1836,green1838}, a technique for obtaining \emph{non-asymptotic} approximations to solutions of linear second-order ODEs with non-constant coefficients. First, we give the intuition behind the method, following the presentation of \citep{olver1997asymptotics}.
	Consider the differential equation
	\begin{equation}\label{eq::GenericODE}
		\ddot{z}(t)-r(t)z(t)=0,\quad t\geq 0,
	\end{equation}
	where $r$ is real-valued, positive, and twice continuously differentiable. Any linear second-order ODE can be reformulated in the form \eqref{eq::GenericODE}, see Lemma~\ref{lem::canonForm} below. Since for all $t\geq 0$, $r(t)\neq 0$, we can use the changes of variables $\tau = \int_0^t \sqrt{r(s)}\diff s$ and $w=r^{1/4}z$ and show that $w$ is solution to
	\begin{equation}\label{eq::GenericODE2}
		\ddot{w}(\tau)-(1+\psi(\tau))w(\tau)=0,\quad t\geq 0,
	\end{equation}
	where\footnote{We express $\psi(\tau)$ via $t$ using the one-to-one correspondence between $\tau$ and $t$ to ease readability.} $\psi(\tau)=\frac{4r(t)r''(t)-5r'(t)^2}{16r(t)^3}$. The LG method consists in neglecting the term $\psi(\tau)$ in \eqref{eq::GenericODE2}, which simply yields two approximate solutions $\hat{w}_1(\tau) = e^{\tau}$ and $\hat w_2(\tau)=e^{-\tau}$. Expressing this in terms of $z$ and $t$, we obtain
	\begin{equation}\label{eq::LG}
		\hat z_1(t) = r(t)^{-1/4}\exp\left(\int_0^t \sqrt{r(s)}\diff s\right) \quad \text{and}\quad \hat z_2(t) = r(t)^{-1/4}\exp\left(\int_0^t -\sqrt{r(s)}\diff s\right).
	\end{equation}
	Those are the LG approximations of the solutions of \eqref{eq::GenericODE}. They are formally valid on any $[0,T]$, $T> 0$ when $\psi$ is ``not too large'' and if $\sqrt{r}$ is integrable on $[0,T]$.

	\begin{remark}
		There exists other (but less intuitive) ways to derive the LG approximations which allow for generalization to higher-order linear ODEs \citep[Chapter~10]{bender1999asymptotic}.
	\end{remark}
	The advantage of this approach is the possibility to estimate the error made using \eqref{eq::LG} w.r.t.\ the true solutions of \eqref{eq::GenericODE}. This is expressed in the following theorem which gathers results from \cite{blumenthal1899uber,olver1961error,taylor1982improved}.
	\begin{theorem}[\cite{olver1997asymptotics}]\label{thm::olver}
		Let $r\colon \R_{+}\to \R$ be a real, positive, twice continuously differentiable function, and define $\varphi(t) = \frac{4r(t)r''(t)-5r'(t)^2}{16r(t)^{5/2}}$ for all $t\geq 0$. Then for any $T>0$, the differential equation,
		\begin{equation}\label{eq::GenericODEthm}
			\ddot{z}(t)-r(t)z(t)=0,\quad t\in [0,T],
		\end{equation}
		has two real and twice continuously differentiable solutions defined $\forall t\in[0,T]$ by,
		\begin{equation*}
			z_1(t) = \frac{1+\delta_1(t)}{r(t)^{1/4}}\exp\left(\int_0^t \sqrt{r(s)}\diff s\right) \ \ \text{and} \ \,
			z_2(t) = \frac{1+\delta_2(t)}{r(t)^{1/4}}\exp\left(-\int_0^t \sqrt{r(s)}\diff s\right),
		\end{equation*}
		where $\displaystyle\vert\delta_1(t)\vert\leq \exp\left(\frac{1}{2} \int_0^t \vert \varphi(s)\vert\diff s\right)-1$ and $\displaystyle\vert\delta_2(t)\vert\leq \exp\left(-\frac{1}{2} \int_t^{T} \vert \varphi(s)\vert\diff s\right)-1$.
		If in addition $\displaystyle\int_0^{+\infty} \vert\varphi(s)\vert\diff s <+\infty$, then the results above also hold for $T=+\infty$.
	\end{theorem}
	\begin{remark}\label{rem::olver} We make the following remarks regarding the above result.\\
		--\ Note that $z_1$ and $z_2$ in Theorem~\ref{thm::olver} are \emph{exact} solutions to \eqref{eq::GenericODEthm}. The LG approximations $\hat z_1$ and $\hat z_2$ are obtained by neglecting the unknown functions $\delta_1$ and $\delta_2$ in $z_1$ and $z_2$. The theorem gives a \emph{non-asymptotic} bound for the errors $\vert z_1(t)-\hat{z}_1(t)\vert$ and $\vert z_2(t)-\hat{z}_2(t)\vert$, $t\geq 0$.
		\\
		--\ Since we assumed $r$ to be twice continuously differentiable and positive, $\varphi$ is continuous, so it is integrable except maybe for $t\to+\infty$.
		\\
		--\ For the sake of simplicity, the formulation of Theorem~\ref{thm::olver} slightly differs from that in \cite{olver1997asymptotics}, the original formulation can be recovered by a change of variable.
	\end{remark}

	\subsection{Liouville--Green Approximation of \texorpdfstring{\eqref{eq::VMDINAVD}}{VM-DIN-AVD}}
	We now proceed to make use of the LG method for approximating the solutions of \eqref{eq::QuadeDINAVD}. The reader only interested in the result can jump directly to the  Section~\ref{sec::mainasymptoticres}. We first make the following assumption.
	\begin{assumption}\label{ass:smooth}
		The functions $\alpha$ and $\eps$ are three times continuously differentiable, and $\eps_0$ is such that $\forall t\geq 0$, $\eps_0<\frac{(\beta\lambda)^2}{2\vert\alpha'(t)\vert+4\lambda}$.
	\end{assumption}
	\begin{remark}
		The condition on $\eps_0$ in Assumption~\ref{ass:smooth} is only technical, so that $r$ defined below is positive.
		It can be easily satisfied since $\vert\alpha'(t)\vert$ is uniformly bounded. Indeed,  $\alpha$ is non-increasing and non-negative (see Section~\ref{sec::pbstatement}), from which one can deduce that $\int_{0}^{+\infty} \vert\alpha'(s)\vert\diff s\leq \alpha_0$.
	\end{remark}
	We now rewrite \eqref{eq::QuadeDINAVD} in the form \eqref{eq::GenericODEthm}.
	\begin{lemma}\label{lem::canonForm}
		Suppose that Assumption~\ref{ass:smooth} holds, and let $x$ be the solution of \ref{eq::QuadeDINAVD}. For all $t\geq 0$, define
		\begin{equation}\label{eq::pr}
			p(t) = \frac{\alpha(t)+\beta\lambda}{\eps(t)} \quad \text{and}  \quad r(t) = \frac{p(t)^2}{4} + \frac{p'(t)}{2} - \frac{\lambda}{\eps(t)}.
		\end{equation}
		Then, $p$ and $r$ are twice continuously differentiable, $r$ is positive and the function $y$ defined for all $t\geq 0$ by $y(t) = x(t)\exp\left(\int_0^t \frac{p(s)}{2}\diff s\right)$ is a solution to
		\begin{equation}\label{eq::ODErewritten}
			\ddot y(t) - r(t)y(t) = 0,\quad t\geq 0,
		\end{equation}
		with initial condition $(y(0),\dot{y}(0) )= (x_0,\dx_0 + \frac{p(0)}{2}x_0)$.
	\end{lemma}
	\begin{proof}
		We first check that for all $t\geq 0$, $r(t)$ is positive. Let $t>0$,
		\begin{equation}\label{eq::rpositive}
			r(t)> 0 \iff \frac{(\alpha(t)+\beta\lambda)^2}{4\eps(t)^2} + \frac{\alpha'(t)}{2\eps(t)} -  \frac{(\alpha(t)+\beta\lambda)\eps'(t)}{\eps(t)^2} - \frac{\lambda}{\eps(t)} >0.
		\end{equation}
		Since $\eps'(t)\leq 0$ and $\alpha'(t)\leq 0$, one can check that a sufficient condition for \eqref{eq::rpositive} to hold is,
		\begin{equation*}
			r(t)> 0 \impliedby \frac{(\alpha(t)+\beta\lambda)^2}{4} > \left(\frac{\vert\alpha'(t)\vert}{2}+\lambda\right)\eps(t) \impliedby \frac{(\beta\lambda)^2}{2\vert\alpha'(t)\vert+4\lambda}>\eps_0.
		\end{equation*}
		So under Assumption~\ref{ass:smooth}, for all $t\geq 0$, $r(t)>0$. We then check that $y$ is indeed solution to \eqref{eq::ODErewritten}. Let $t>0$,
		\begin{equation*}
			\dot{y}(t) = \frac{p(t)}{2}x(t)\exp\left(\int_0^t \frac{p(s)}{2}\diff s\right) + \dx(t)\exp\left(\int_0^t \frac{p(s)}{2}\diff s\right),\quad \text{and,}
		\end{equation*}
		\begin{equation*}
			\ddot{y}(t) = \exp\left(\int_0^t \frac{p(s)}{2}\diff s\right)\left[\left(\frac{p(t)^2}{4}+\frac{p'(t)}{2}\right)x(t) + p(t)\dx(t) + \ddx(t)\right].
		\end{equation*}
		Since $x$ solves \eqref{eq::QuadeDINAVD}, it holds that $\ddx(t) = -p(t)\dx(t) - \frac{\lambda}{\eps(t)}x(t)$, so,
		\begin{multline*}
			\ddot{y}(t) = \exp\left(\int_0^t \frac{p(s)}{2}\diff s\right)\left(\frac{p(t)^2}{4}+\frac{p'(t)}{2}- \frac{\lambda}{\eps(t)}\right)x(t) \\= \left(\frac{p(t)^2}{4}+\frac{p'(t)}{2}- \frac{\lambda}{\eps(t)}\right)y(t)
			= r(t)y(t).
		\end{multline*}
	\end{proof}
	Lemma~\ref{lem::canonForm} gives a reformulation of \eqref{eq::QuadeDINAVD} suited to apply Theorem~\ref{thm::olver}.
	To use the theorem for all $t\geq 0$, we need to ensure that $\varphi(t)=\frac{4r(t)r''(t)-5r'(t)^2}{16r(t)^{5/2}}$ is integrable. To this aim we make the following assumption.
	\begin{assumption}\label{ass:subexp}
		The functions $\eps$ and $\alpha$ have first, second and third-order derivatives that are integrable on $[0,+\infty[$. In addition, $\lim\limits_{t\to\infty}{\eps(t)}=0$ and $\eps'(t)^2/\eps(t)$ is integrable on $[0,+\infty[$.
	\end{assumption}
	\begin{remark}
		Assumption~\ref{ass:subexp} holds for most decays used in practice, with in particular any polynomial decay of the form $\frac{\eps_0}{(t+1)^a}$ and $\frac{\alpha_0}{(t+1)^b}$, $a\in\N\setminus \{0\}$ and $b\in\N$. Note that $\eps$ and $\alpha$ need not be integrable and $\alpha$ can even be constant.
	\end{remark}
	The next lemma states the integrability of $\varphi$ on $[0,+\infty[$.
	\begin{lemma}\label{lem::phiintegrable}
		Under Assumption~\ref{ass:smooth} and~\ref{ass:subexp}, $\int_0^{+\infty} \vert \varphi(s)\vert\diff s<+\infty$.
	\end{lemma}
	The proof of this lemma involves long computations and is thus postponed to Appendix~\ref{app::integrability}.
	We can now use Theorem~\ref{thm::olver} to obtain an exact form for the solution of \eqref{eq::QuadeDINAVD} based on the LG approximations.
	\begin{theorem}\label{thm::MainResLG}
		Suppose that Assumptions~\ref{ass:smooth} and~\ref{ass:subexp} hold. There exists $A,B\in\R$ such that $x(0)=x_0$, $\dx(0)=\dx_0$ and for all $t\geq 0$, the solution of \eqref{eq::QuadeDINAVD} is
		\begin{align}\label{eq::formofx}
			\begin{split}
				x(t) =&
				A\frac{1+\delta_1(t)}{r(t)^{1/4}}\frac{\sqrt{\alpha(t)+\beta\lambda}}{\sqrt{\eps(t)}}
				\exp\left(\int_0^t -\frac{\lambda}{\alpha(s)+\beta\lambda} - \frac{\lambda^2\eps(s)}{(\alpha(s)+\beta\lambda)^{3}}+o(\eps(s))\diff s\right)
				\\
				+B&\frac{1+\delta_2(t)}{r(t)^{1/4}}
				\frac{\sqrt{\eps(t)}}{\sqrt{\alpha(t)+\beta\lambda}}
				\\
				&\exp\left(\int_0^t -\frac{\alpha(s)+\beta\lambda}{\eps(s)}+\frac{\lambda}{\alpha(s)+\beta\lambda}+\frac{\lambda^2\eps(s)}{(\alpha(s)+\beta\lambda)^{3}}+o(\eps(s))\diff s\right),
			\end{split}
		\end{align}
		where for all  $t\geq 0$,
		\begin{equation}\label{eq::deltabounds}
			\vert\delta_1(t)\vert \leq \exp\left(\frac{1}{2} \int_0^t \vert \varphi(s)\vert\diff s\right)-1
			\quad \text{and}\quad
			\vert\delta_2(t)\vert \leq \exp\left(-\frac{1}{2} \int_t^{+\infty} \vert \varphi(s)\vert\diff s\right)-1.
		\end{equation}
	\end{theorem}
	Thanks to the bounds \eqref{eq::deltabounds}, we now have an approximation of $x$.  We will use it in particular to compare $x$ asymptotically to the solutions of \eqref{eq::quadLM} and \eqref{eq::quadCN}.
	Before this, we prove Theorem~\ref{thm::MainResLG}.

	\begin{proof}[Proof of Theorem~\ref{thm::MainResLG}]
		Let $x$ be the solution of \eqref{eq::QuadeDINAVD} define $p,r$ as in \eqref{eq::pr}. Let us also define $y(t)\stackrel{\textrm{def}}{=}x(t)\exp\left(\int_0^t \frac{p(s)}{2}\diff s\right)$. According to Lemma~\ref{lem::canonForm}, $r$ is positive and $y$ is solution to \eqref{eq::ODErewritten}. Then, from Lemma~\ref{lem::phiintegrable}, $\int_t^{T} \vert \varphi(s)\vert\diff s<+\infty$, so we can apply Theorem~\ref{thm::olver} to $y$ on $[0,+\infty[$. Therefore, there exists $A,B\in\R$, such that $\forall t\geq 0$,
		$$
		y(t)= A\frac{1+\delta_1(t)}{r(t)^{1/4}}\exp\left(\int_0^t \sqrt{r(s)}\diff s\right) + B\frac{1+\delta_2(t)}{r(t)^{1/4}}\exp\left(\int_0^t -\sqrt{r(s)}\diff s\right),
		$$
		where $A$ and $B$ are determined by the initial conditions, and $\delta_1$, $\delta_2$ are such that \eqref{eq::deltabounds} holds.
		Going back to $x(t) = y(t)\exp\left(\int_0^t -\frac{p(s)}{2}\diff s\right)$, we obtain 	that for all $t\geq 0$,
		\begin{multline}\label{eq::formofx0}
			x(t) = A\frac{1+\delta_1(t)}{r(t)^{1/4}}\exp\left(\int_0^t -\frac{p(s)}{2} +\sqrt{r(s)}\diff s\right) \\ + B\frac{1+\delta_2(t)}{r(t)^{1/4}}\exp\left(\int_0^t -\frac{p(s)}{2} -\sqrt{r(s)}\diff s\right).
		\end{multline}
		It now remains to expand the terms in the two exponentials in \eqref{eq::formofx0} in order to obtain \eqref{eq::formofx}. To this aim, we approximate $\sqrt{r(s)}$, let $s\geq 0$,
		\begin{align}\label{eq::sqrtr}
			\begin{split}
				\sqrt{r(s)}&=\frac{p(s)}{2}\sqrt{1+ \frac{2p'(s)}{p(s)^2} - \frac{4\lambda}{\eps(s)p(s)^2}}
				\\&= \frac{p(s)}{2}\left(1 + \frac{p'(s)}{p(s)^2} - \frac{2\lambda}{\eps(s)p(s)^2} - \frac{1}{8}\left(\frac{2p'(s)}{p(s)^2} - \frac{4\lambda}{\eps(s)p(s)^2}\right)^2 + o(\eps(s)^2) \right)
				\\ & = \frac{p(s)}{2} +\frac{ p'(s)}{2p(s)} - \frac{\lambda}{\eps(s)p(s)} - \frac{1}{16}\left(\frac{2p'(s)}{p(s)^{3/2}} - \frac{4\lambda}{\eps(s)p(s)^{3/2}}\right)^2 + o(\eps(s))
				\\  = \frac{p(s)}{2} &+ \frac{p'(s)\eps(s)}{2(\alpha(s)+\beta\lambda)} - \frac{\lambda}{\alpha(s)+\beta\lambda} - \frac{1}{16}\left(\frac{2p'(s)}{p(s)^{3/2}} - \frac{4\lambda\sqrt{\eps(s)}}{(\alpha(s)+\beta\lambda)^{3/2}}\right)^2 + o(\eps(s))
				\\  = \frac{p(s)}{2} &+ \frac{\alpha'(s)/2-\lambda}{\alpha(s)+\beta\lambda} -\frac{\eps'(s)}{2\eps(s)}  - \frac{1}{16}\left(\frac{2p'(s)}{p(s)^{3/2}} - \frac{4\lambda\sqrt{\eps(s)}}{(\alpha(s)+\beta\lambda)^{3/2}}\right)^2 + o(\eps(s))
			\end{split}
		\end{align}
		To ease the readability, we denote $h(t)=\left(\frac{2p'(s)}{p(s)^{3/2}} - \frac{4\lambda\sqrt{\eps(s)}}{(\alpha(s)+\beta\lambda)^{3/2}}\right)^2$. Focusing on the first exponential term in \eqref{eq::formofx0}, we deduce from \eqref{eq::sqrtr} that for all $t\geq 0$,
		\begin{align*}
			\begin{split}
				&\exp\left(\int_0^t -\frac{p(s)}{2} +\sqrt{r(s)}\diff s\right)
				\\&= \exp\left(\int_0^t \frac{\alpha'(s)/2}{\alpha(s)+\beta\lambda} -\frac{\eps'(s)}{2\eps(s)} -\frac{\lambda}{\alpha(s)+\beta\lambda} - \frac{1}{16}h(t) + o(\eps(s))\diff s\right)
				\\&=  \frac{\sqrt{\alpha(t)+\beta\lambda})}{\sqrt{\alpha_0+\beta\lambda}} \frac{\sqrt{\eps_0}}{\sqrt{\eps(t)}} \exp\left(
				\int_0^t \frac{-\lambda}{\alpha(s)+\beta\lambda} -
				\frac{1}{16}h(t) + o(\eps(s))\diff s\right)
				\\&=  \frac{\sqrt{\alpha(t)+\beta\lambda})}{\sqrt{\alpha_0+\beta\lambda}} \frac{\sqrt{\eps_0}}{\sqrt{\eps(t)}} \exp\left(
				\int_0^t \frac{-\lambda}{\alpha(s)+\beta\lambda} -
				\frac{\lambda^2\eps(s)}{(\alpha(s)+\beta\lambda)^{3}} + o(\eps(s))\diff s\right),
			\end{split}
		\end{align*}
		where the last line relies on further computations postponed to Lemma~\ref{lem::remainderterm} in Appendix~\ref{app::integrability}. Performing the exact same type of computations on $\exp\left(\int_0^t -\frac{p(s)}{2} -\sqrt{r(s)}\diff s\right)$, and up to redefining $A$ and $B$ so as to encompass all the constants, we obtain \eqref{eq::formofx} and the result is proved.
	\end{proof}

	\subsection{Comparison of \texorpdfstring{$x$}{x} with \texorpdfstring{$x_{LM}$}{xLM} and \texorpdfstring{$x_N$}{xN}}\label{sec::mainasymptoticres}

	We now have an expression for $x$ which is almost explicit: we do not know $\delta_1$ and $\delta_2$ in closed form, but they are uniformly bounded (by Lemma~\ref{lem::phiintegrable}). We will now  compare the asymptotic behavior of \eqref{eq::formofx} with those of the solutions of \eqref{eq::quadLM} and \eqref{eq::quadCN} that we denoted $x_{LM}$ and $x_N$ respectively.
	Our main result of Section~\ref{sec::Quads} is the following, where $\sim_{+\infty}$ denotes the asymptotic equivalence\footnote{Two real-valued functions $g_1$ and $g_2$ are asymptotically equivalent in $+\infty$ if and only if $\lim_{t\to\infty}\frac{g_1(t)}{g_2(t)}=1$.} between two functions as $t\to\infty$.
	\begin{theorem}\label{thm::MainResAsymptotic}
		Let $x$ be the solution of \eqref{eq::QuadeDINAVD}, given in \eqref{eq::formofx}, and $x_{LM}$ and $x_N$ whose closed forms are stated in \eqref{eq::closeformLMN}. Under Assumptions~\ref{ass:smooth} and~\ref{ass:subexp}, there exists $C>0$ such that the following asymptotic equivalences hold:
		\begin{align}\label{eq::equivx}
			\begin{split}
				x(t) &\sim_{+\infty}
				x_{LM}(t)C\exp\left(\int_0^t - \frac{\lambda^2\eps(s)}{(\alpha(s)+\beta\lambda)^{3}}+o(\eps(s))\diff s\right),\quad \text{and}
				\\x(t)&\sim_{+\infty}
				x_{N}(t)C\exp\left(\int_0^t \frac{\alpha(s)}{\beta(\alpha(s)+\beta\lambda)}- \frac{\lambda^2\eps(s)}{(\alpha(s)+\beta\lambda)^{3}}+o(\eps(s))\diff s\right).
			\end{split}
		\end{align}
		As a consequence, the convergence of $x$ to $x^\star$ is:
		\begin{enumerate}[label=(\roman*)]
			\item \emph{Faster} than that of $x_{LM}$
			if $\eps$ is non-integrable and as fast otherwise.
			\item Slower than that of $x_{N}$ if $\alpha$ is non-integrable and as fast if $\alpha$ is integrable, in the case where $\forall t\geq 0$, $\alpha(t)>\eps(t)$.
			\item \emph{Faster} than that of $x_{N}$ if $\eps$ is non-integrable and as fast if $\eps$ is integrable, in the case where $\forall t\geq 0$, $\alpha(t)< \eps(t)$.
		\end{enumerate}
	\end{theorem}

	While the results of Section~\ref{sec::Control} were related to the closeness of \eqref{eq::VMDINAVD} w.r.t.\ \eqref{eq::CN} and \eqref{eq::LM} from a control perspective, Theorem~\ref{thm::MainResAsymptotic} provides a different type of insight. First, the results are asymptotic, so they only allow controlling \eqref{eq::VMDINAVD} for large $t$. They provide however a clear understanding of the nature of the solutions of \eqref{eq::VMDINAVD} and their convergence. The conclusions (summarized in Table~\ref{tabl::speedsummary}) are in accordance with what we would expect: when the viscous damping is larger than the variable mass,
	\eqref{eq::VMDINAVD} behaves more like the Levenberg--Marquardt method than the Newton one, but it actually becomes an accelerated Levenberg--Marquardt dynamics when $\eps$ is non-integrable but vanishing. However, when the variable mass $\eps$ is larger than $\alpha$, the dynamics is closer to the one of the Newton method, and can actually be an accelerated Newton dynamics, again for non-integrable $\eps$. This is analogous to the necessary condition that $\alpha$ must be non-integrable in order to accelerate first-order methods in convex optimization (see \cite{attouch2017asymptotic}). We conclude this section by proving Theorem~\ref{thm::MainResAsymptotic}.

	\begin{proof}[Proof of Theorem~\ref{thm::MainResAsymptotic}]
		Thanks to Assumptions~\ref{ass:smooth} and~\ref{ass:subexp}, Theorem~\ref{thm::MainResLG} tells us that $x$ has the form \eqref{eq::formofx}. We now analyze the two terms in \eqref{eq::formofx}.

		First, we know from Theorem~\ref{thm::MainResLG} that $\delta_1(0)=0$ and $\lim_{t\to+\infty}\delta_2(t)=0$. In addition, by Lemma~\ref{lem::phiintegrable}, $\delta_1$ and $\delta_2$ are uniformly bounded by some positive constant.
		Then $r(t)^{-1/4}$ decays asymptotically like $\sqrt{\eps(t)}$ and $\alpha$ is bounded. So
		$A\frac{1+\delta_1(t)}{r(t)^{1/4}}\frac{\sqrt{\alpha(t)+\beta\lambda}}{\sqrt{\eps(t)}}$ is asymptotically equivalent to some constant $c_1\in\R$ as $t\to +\infty$. Similarly, the factor $B\frac{1+\delta_2(t)}{r(t)^{1/4}}\frac{\sqrt{\eps(t)}}{\sqrt{\alpha(t)+\beta\lambda}}$ is equivalent to $c_2\eps(t)$, with $c_2\in\R$.

		We now analyze the ``exponential factors'' in \eqref{eq::formofx}.
		On the one hand, $\frac{\lambda}{\alpha(s)+\beta\lambda}+\frac{\lambda^2\eps(s)}{(\alpha(s)+\beta\lambda)^{3}}+o(\eps(s))$ converges to $\frac{1}{\beta}$ as $s\to \infty$, while $\frac{\alpha(s)+\beta\lambda}{\eps(s)}$ diverges to $+\infty$. Therefore, we deduce that,
		\begin{multline*}
			\exp\left(\int_0^t -\frac{\alpha(s)+\beta\lambda}{\eps(s)}+\frac{\lambda}{\alpha(s)+\beta\lambda}+\frac{\lambda^2\eps(s)}{(\alpha(s)+\beta\lambda)^{3}}+o(\eps(s))\diff s\right)
			\\= o\left(\exp\left(\int_0^t -\frac{\lambda}{\alpha(s)+\beta\lambda} - \frac{\lambda^2\eps(s)}{(\alpha(s)+\beta\lambda)^{3}}+o(\eps(s))\diff s\right)\right).
		\end{multline*}
		As a consequence, the second term in \eqref{eq::formofx} will decrease to $0$ faster than the first-one (let alone the additional $\eps(t)$ decay that we have just discussed). The asymptotic behavior of $x$ will thus be governed by the first term in \eqref{eq::formofx}.

		Let us now focus on the first term in \eqref{eq::formofx}. Observe that
		$\exp\left(\int_0^t -\frac{\lambda}{\alpha(s)+\beta\lambda}\diff s\right)$ is exactly the decay of $x_{LM}$ in \eqref{eq::closeformLMN}. Thus, we have proved that there exists $C>0$, such that the following asymptotic equivalence holds,
		\begin{multline*}
			A\frac{1+\delta_1(t)}{r(t)^{1/4}}\frac{\sqrt{\alpha(t)+\beta\lambda}}{\sqrt{\eps(t)}}\exp\left(\int_0^t -\frac{\lambda}{\alpha(s)+\beta\lambda} - \frac{\lambda^2\eps(s)}{(\alpha(s)+\beta\lambda)^{3}}+o(\eps(s))\diff s\right)
			\\ \sim_{+\infty}
			x_{LM}(t)C\exp\left(\int_0^t - \frac{\lambda^2\eps(s)}{(\alpha(s)+\beta\lambda)^{3}}+o(\eps(s))\diff s\right),
		\end{multline*}
		which proves the first part of \eqref{eq::equivx}. The second equivalence in \eqref{eq::equivx} is obtained using the following identity,
		\begin{equation}\label{eq::rewritteLM}
			\int_0^t -\frac{\lambda}{\alpha(s)+\beta\lambda}\diff s = \int_0^t -\frac{1}{\beta} + \frac{\alpha(s)}{\beta(\alpha(s)+\beta\lambda)}\diff s = -\frac{t}{\beta} + \int_0^t  \frac{\alpha(s)}{\beta(\alpha(s)+\beta\lambda)}\diff s
		\end{equation}
		and $e^{-t/\beta}$ is precisely the rate at which $x_N$ decreases. So \eqref{eq::equivx} holds.

		It finally remains to deduce the conclusions of the theorem from \eqref{eq::equivx}.
		\\--\ Regarding the comparison with $x_{LM}$, the integral $\int_0^t - \frac{\lambda^2\eps(s)}{(\alpha(s)+\beta\lambda)^{3}}+o(\eps(s))\diff s$ converges if and only if $\eps$ is integrable on $[0,+\infty[$, and diverges to $-\infty$ when $\eps$ is not. So $x$ converges to $0$ at least as fast as $x_{LM}$ and faster when $\eps$ is not integrable.
		\\-- As for the comparison with $x_N$, if $\alpha(s)>\eps(s)\geq 0$ for all $s\geq 0$, then the integral $\int_0^t \frac{\alpha(s)}{\beta(\alpha(s)+\beta\lambda)}- \frac{\lambda^2\eps(s)}{(\alpha(s)+\beta\lambda)^{3}}+o(\eps(s))\diff s$ is convergent in $+\infty$ if and only if $\alpha$ is integrable and diverges to $+\infty$ when $\alpha$ is non-integrable. So when $\alpha$ is integrable, the speed of convergence of $x$ is the same as that of $x_N$. When $\alpha$ is not integrable, the convergence to $0$ is slower but still holds. Indeed, for all $s\geq 0$
		$\frac{\alpha(s)}{\beta(\alpha(s)+\frac{1}{\beta})}< \frac{\alpha(s)}{\beta\alpha(s)}= \frac{1}{\beta}$. Thus for all $t>0$, $-\frac{t}{\beta} + \int_0^t \frac{\alpha(s)}{\beta(\alpha(s)+\beta\lambda)}\diff s<0$.
		\\-- Finally, the comparison with $x_N$ in the case $\eps(s)> \alpha(s)$ is exactly the same as the comparison with $x_{LM}$ using \eqref{eq::rewritteLM}.
	\end{proof}

	\begin{remark}
		As discussed in the beginning of this section, the main motivation for studying quadratic functions is that strongly convex functions with Lipschitz continuous gradient can be tightly both upper- and lower-bounded by quadratic functions around their minimizers. The intuitive attempt to extend our analysis to this class of functions consists therefore in applying Theorem~\ref{thm::MainResLG} to the upper and lower quadratic bounds. Such analysis is hence only possible for initial conditions close-enough to the minimizer, and even in that case we did not manage to derive precise-enough asymptotic estimates for the solution of \eqref{eq::VMDINAVD} using this strategy. Whether this is possible or not remains an open question.
	\end{remark}

	\section{Numerical Experiments}\label{sec::exp}
	\begin{figure}[htb]
		\begin{minipage}{\linewidth}
			\centering
			\includegraphics[width=0.7\linewidth]{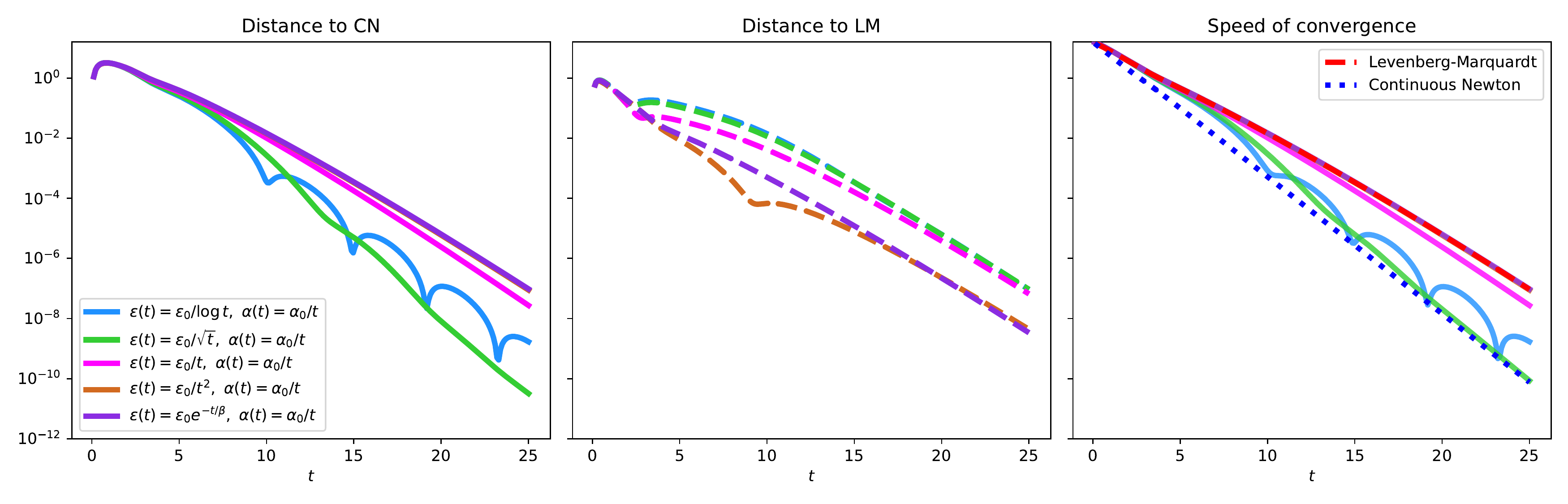}
			\\
			\includegraphics[width=0.7\linewidth]{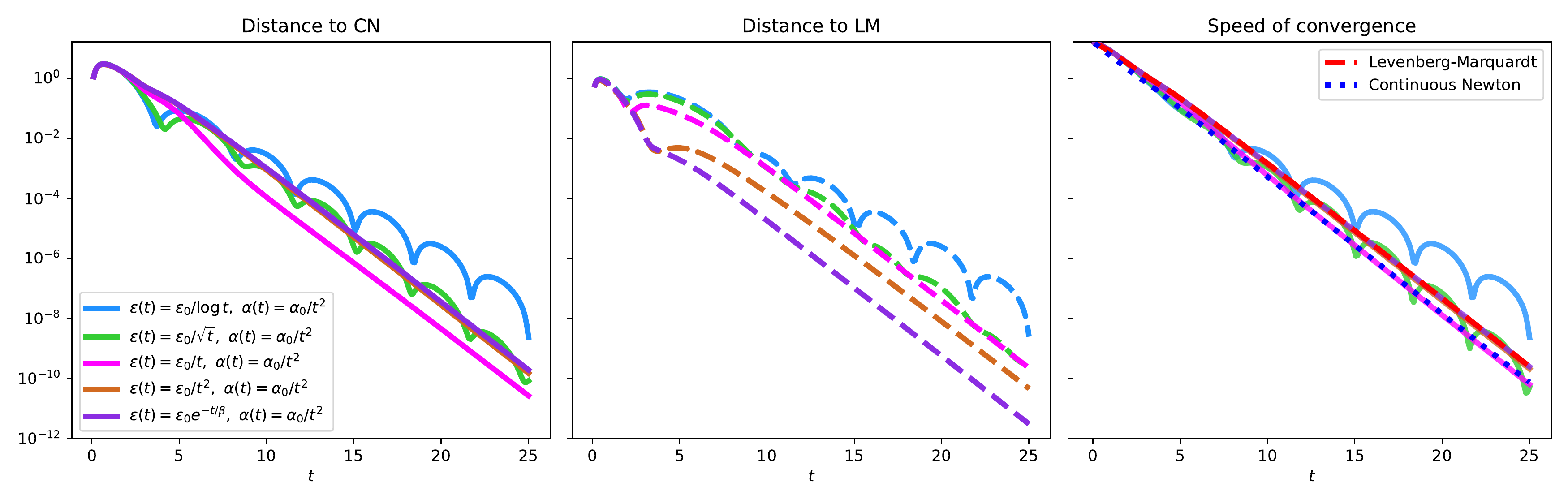}
		\end{minipage}
		\caption{\label{fig::expmt}
			Comparison of the solutions $x_N$, $x_{LM}$ and $x$ of \eqref{eq::CN}, \eqref{eq::LM} and \eqref{eq::VMDINAVD} respectively, for a strongly convex function of the form $f(x)=e^{-\Vert x\Vert^2} + \frac 1 2 \norm{Ax}^2 $.
			Left figures: distance $\norm{x(t)-x_N(t)}$ versus time $t$, each curve corresponds to a different choice of $\eps$; middle figures: distance $\norm{x(t)-x_{LM}(t)}$, again for several $\eps$. Right figures: distance to the optimum $x^\star$ for reference, $x_N$ and $x_{LM}$ are in dotted and dashed lines, other curves correspond to \eqref{eq::VMDINAVD} for several choices of $\eps$. The brown curve is often hidden behind the purple (and sometimes the pink) curve.
			Top and bottom rows show results respectively for non-integrable and integrable viscous dampings $\alpha$.
			The theoretical bounds from Theorem~\ref{thm::mainGenResult} are only displayed on Figure~\ref{fig::LP} below, for the sake of readability.}
	\end{figure}

	\begin{figure}[htb]
		\begin{minipage}{\linewidth}
			\centering
			\includegraphics[width=0.7\linewidth]{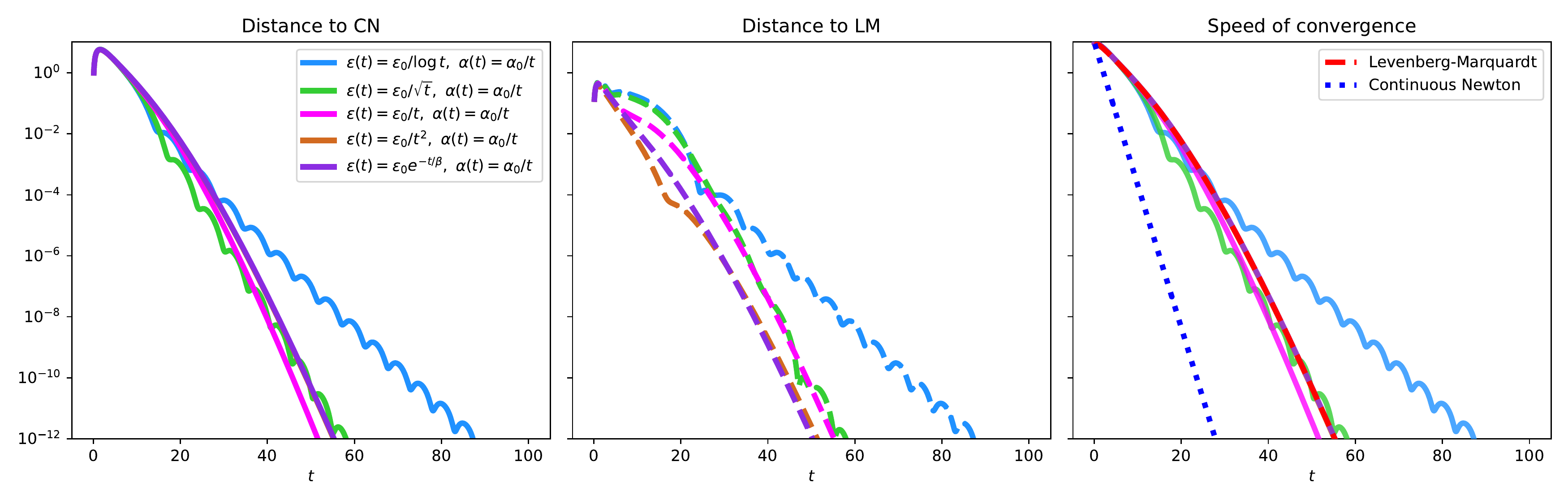}
			\includegraphics[width=0.7\linewidth]{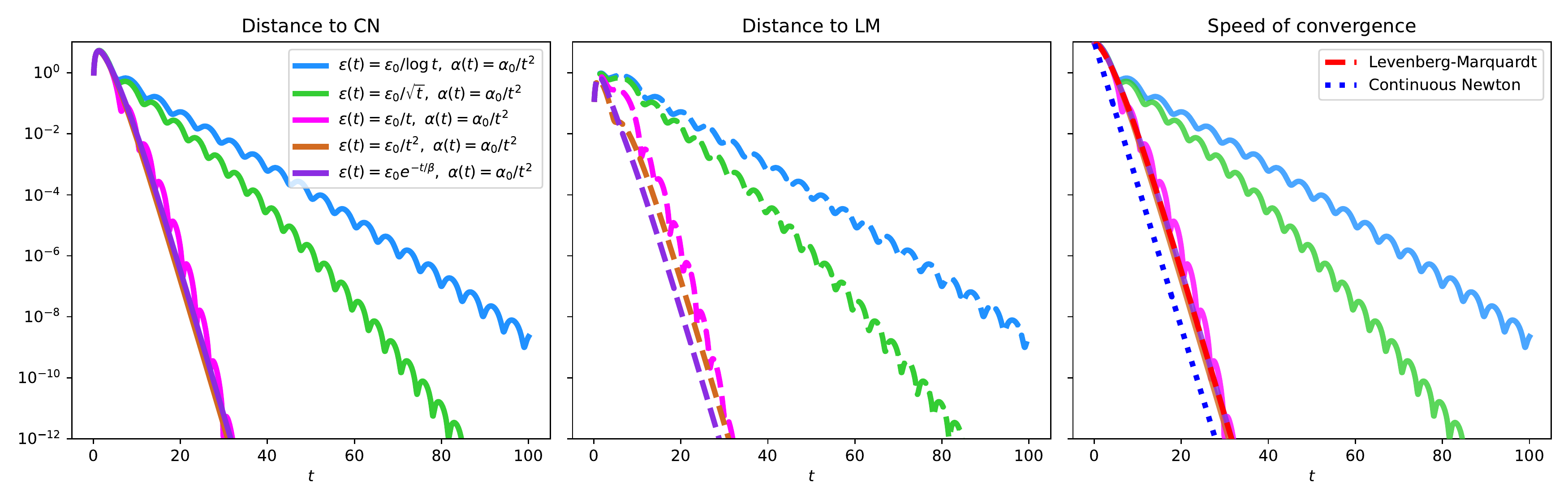}
		\end{minipage}
		\caption{\label{fig::logsumexpexp} Similar experiment and figures as those described in Figure~\ref{fig::expmt}, but for the function $f(x)=\log\left(\sum_{i=1}^n e^{x_i} + e^{-x_i}\right) + \frac 1 2 \norm{Ax}^2 $.}
	\end{figure}

	\begin{figure}[htb]
		\begin{minipage}{\linewidth}
			\centering
			\includegraphics[width=0.7\linewidth]{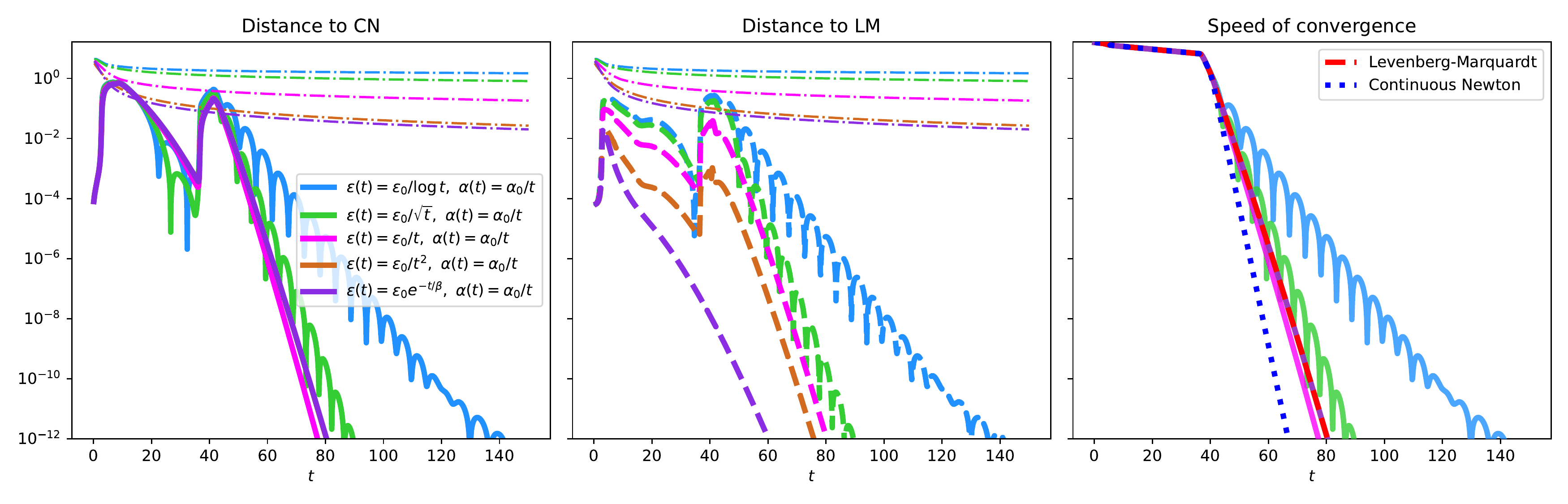}
			\\
			\includegraphics[width=0.7\linewidth]{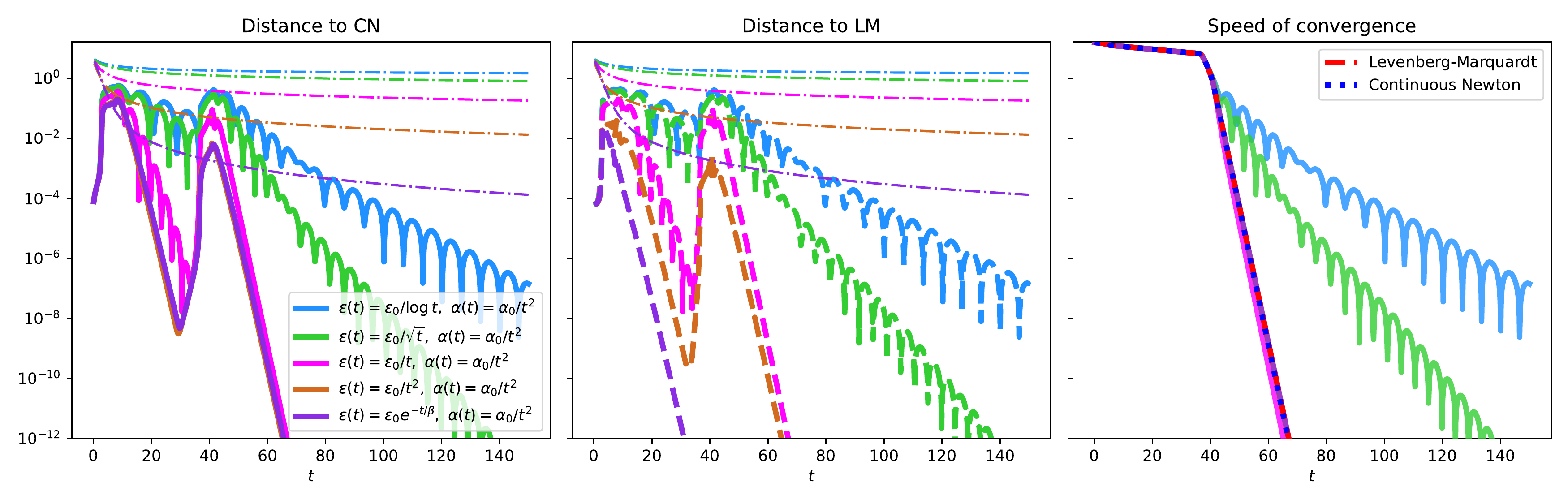}
		\end{minipage}
		\caption{\label{fig::LP} Similar experiment and figures as those described in Figure~\ref{fig::expmt}, but for the function $f(x)=\sum_{i=1}^n x_i^{50} + \frac 1 2 \norm{Ax}^2 $.
			The thin ``dash dotted'' curves represent approximations of the theoretical bounds from Theorem~\ref{thm::mainGenResult} for each choice of $(\eps,\alpha)$ considered.}
	\end{figure}

	\begin{figure}[htb]
		\centering
		\includegraphics[width=0.25\linewidth]{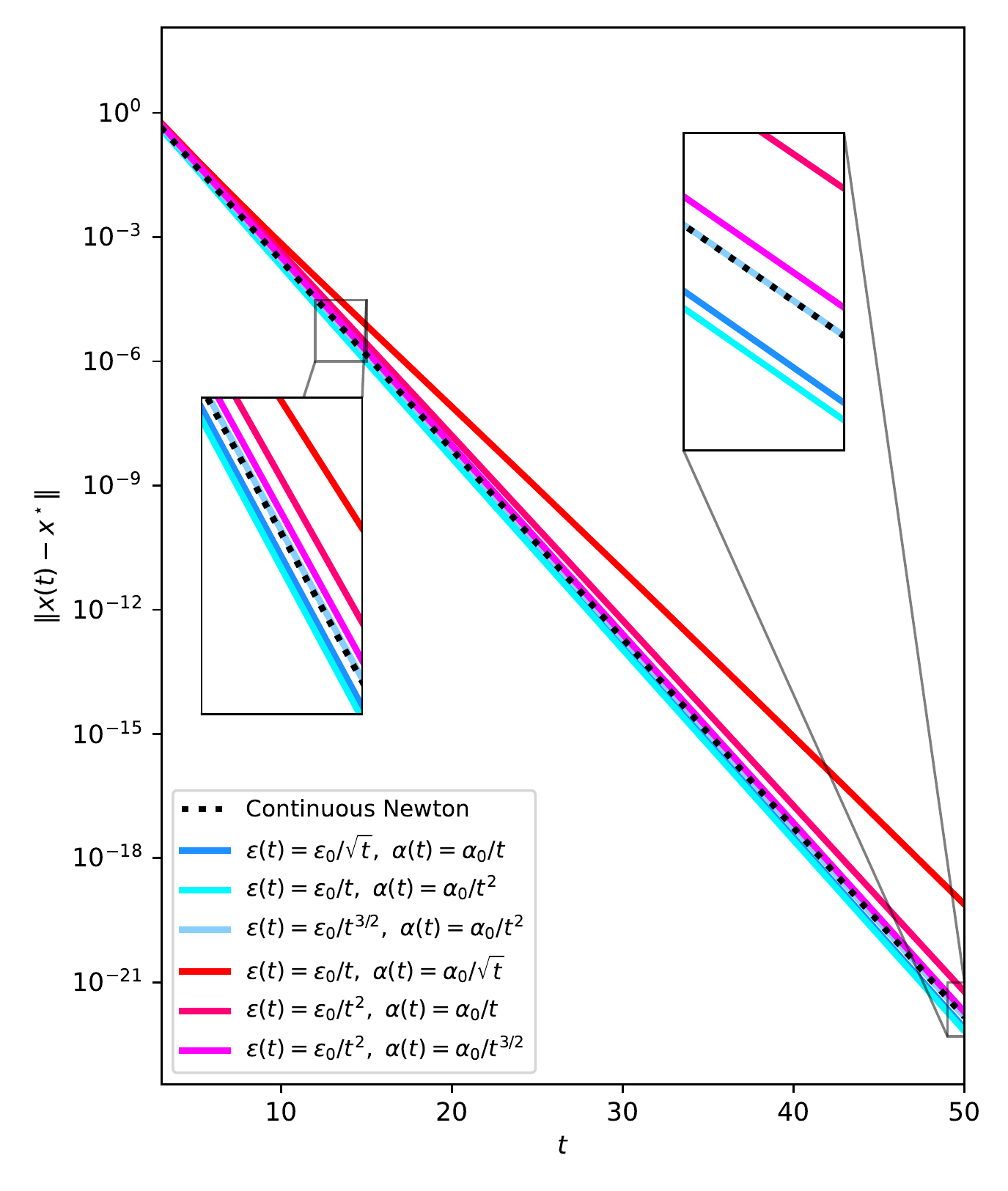} \hspace{0.1cm}\includegraphics[width=0.25\linewidth]{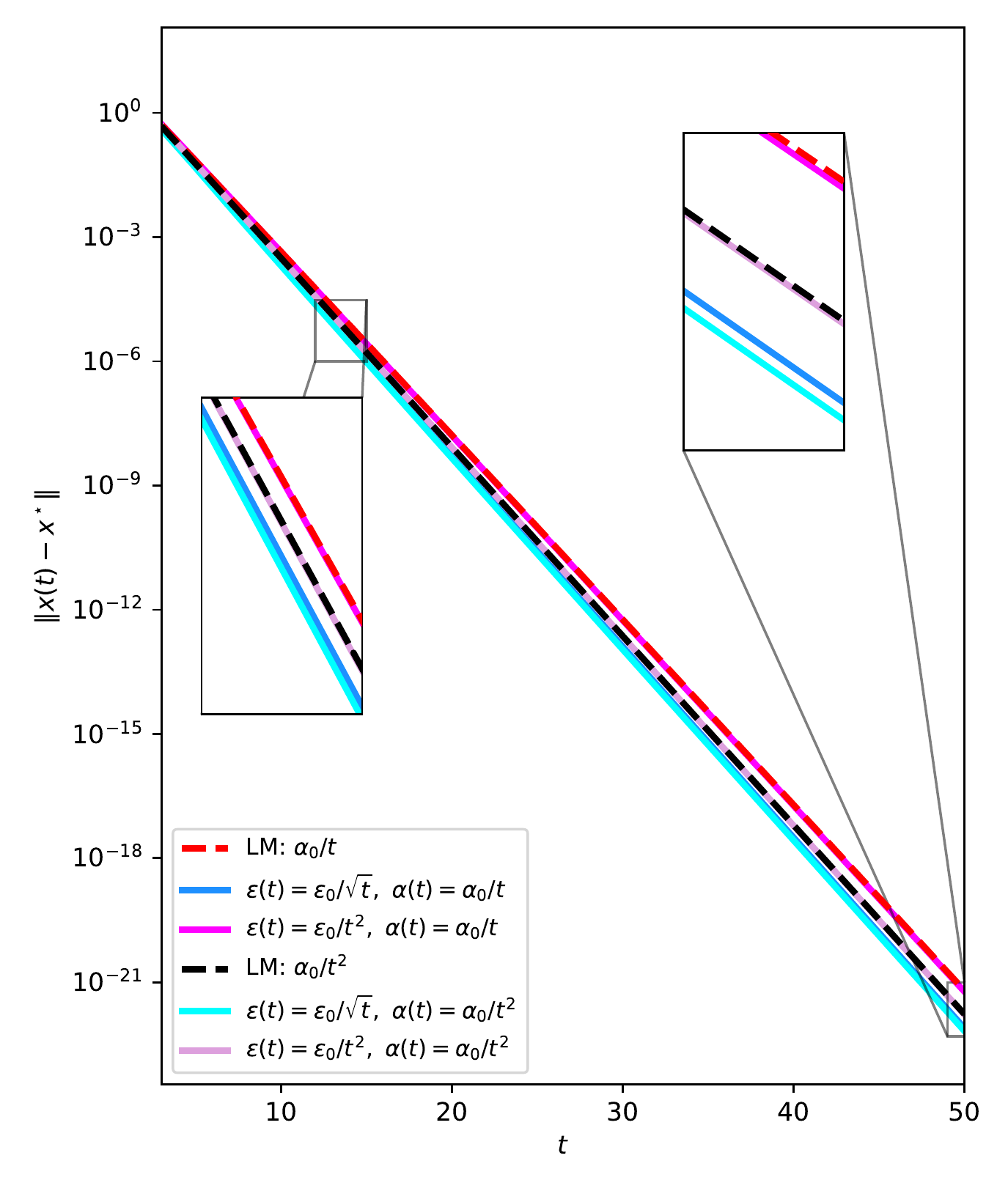}
		\caption{\label{fig::quadspeed} Numerical validation of Theorem~\ref{thm::MainResAsymptotic}: distance to the optimum $x^\star$ as a function of time on a quadratic function $f(x)=\frac{1}{2}\norm{Ax}^2$. Left: speed comparison w.r.t.\ \eqref{eq::CN} for several choices of $\eps$ and $\alpha$. Right: Comparison with LM for $\alpha$ integrable or not and several choices of $\eps$. Shades of blue represent cases where $\eps(t)>\alpha(t)$ while shades of red represent the opposite setting.}
	\end{figure}

	\begin{figure}[htb]
		\centering
		\includegraphics[width=0.24\linewidth]{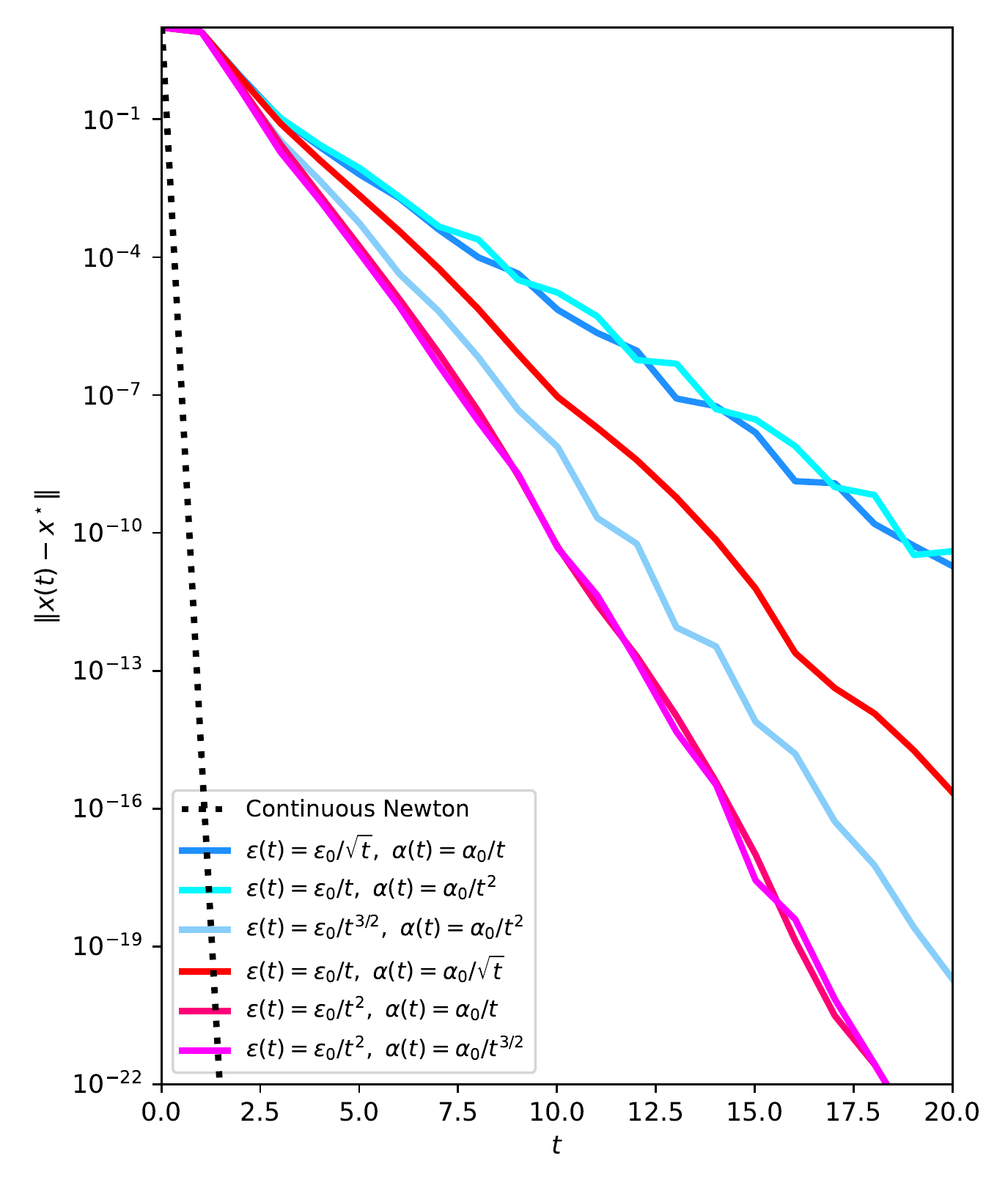}
		\includegraphics[width=0.24\linewidth]{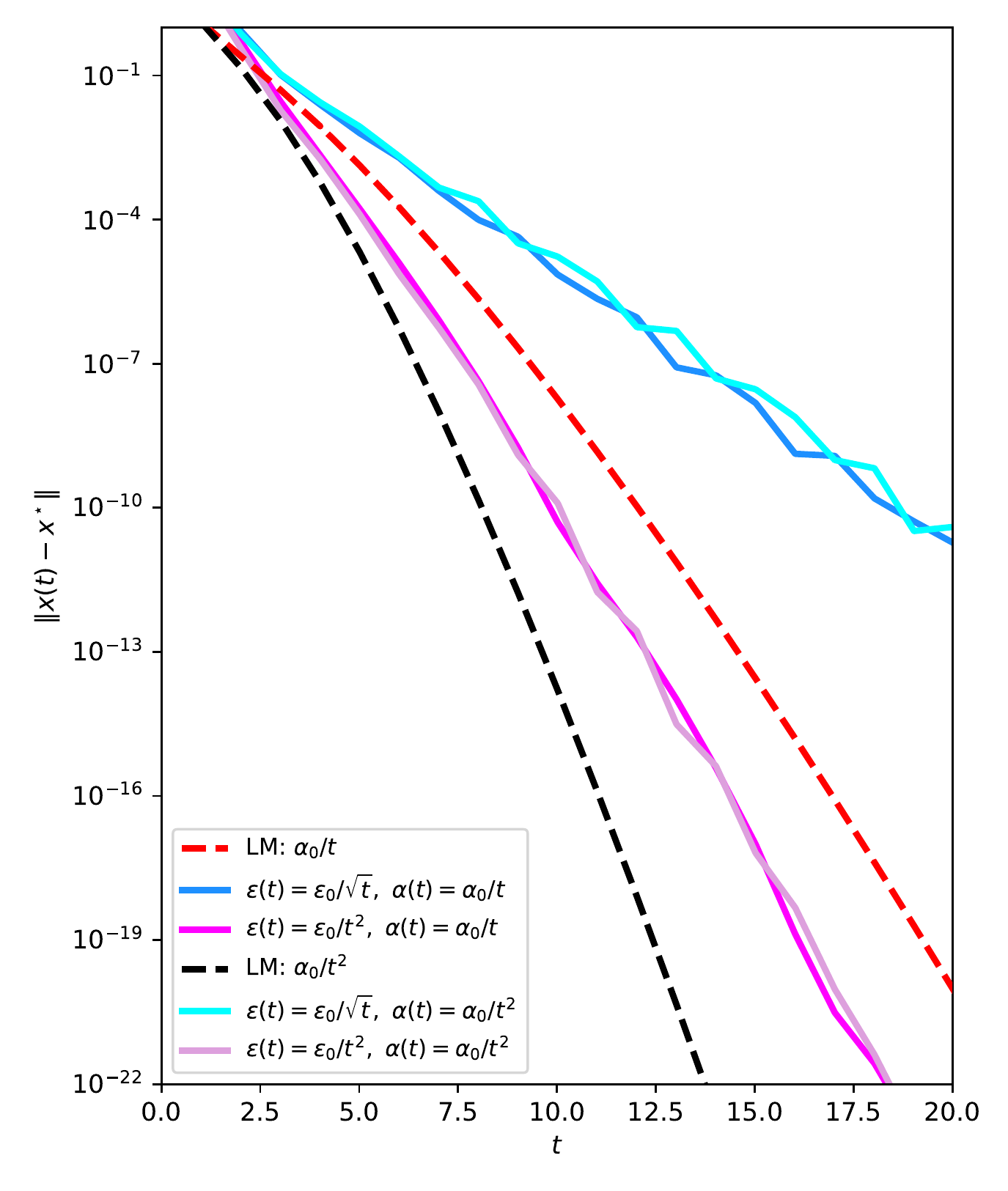}
		\includegraphics[width=0.24\linewidth]{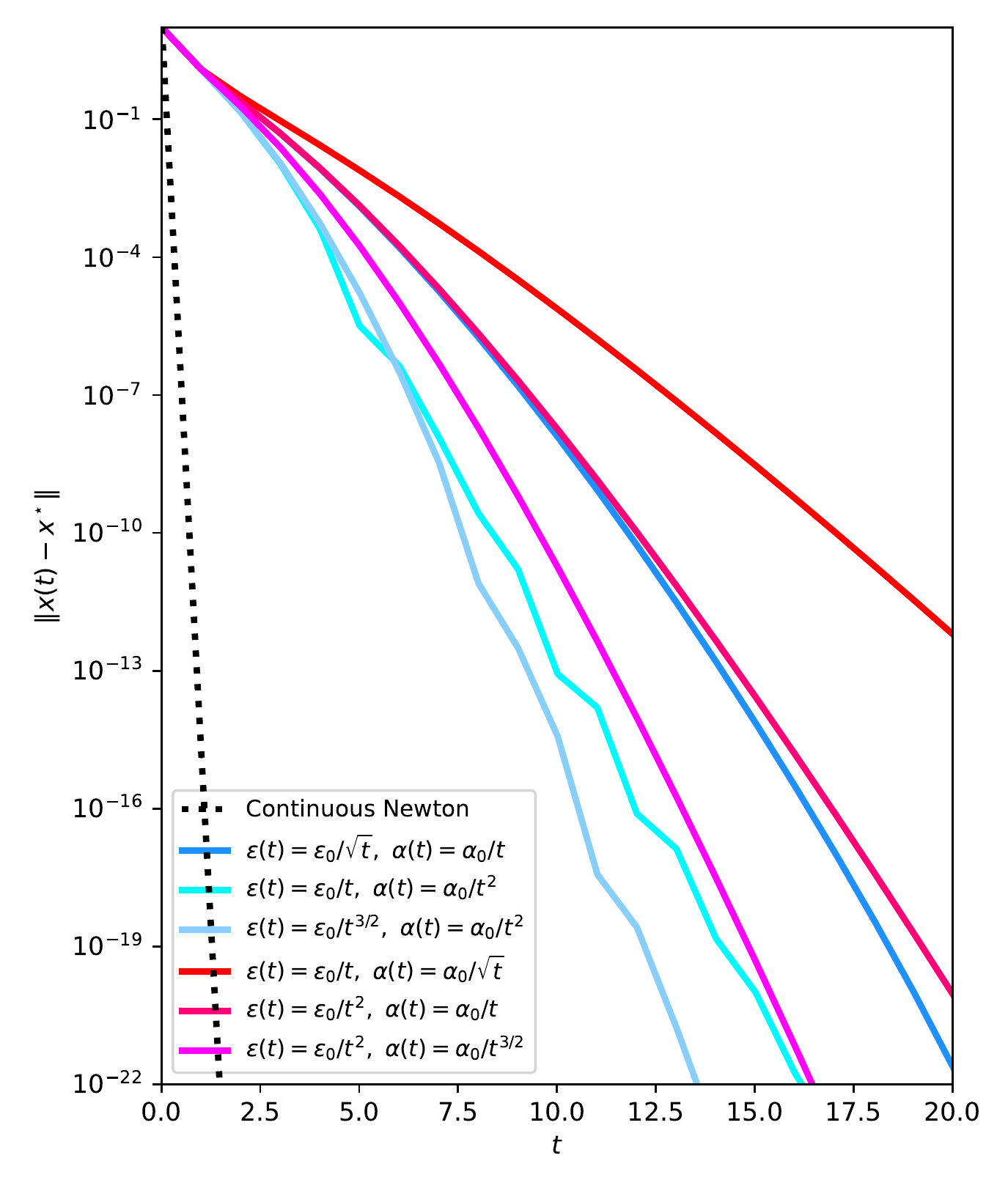}
		\includegraphics[width=0.24\linewidth]{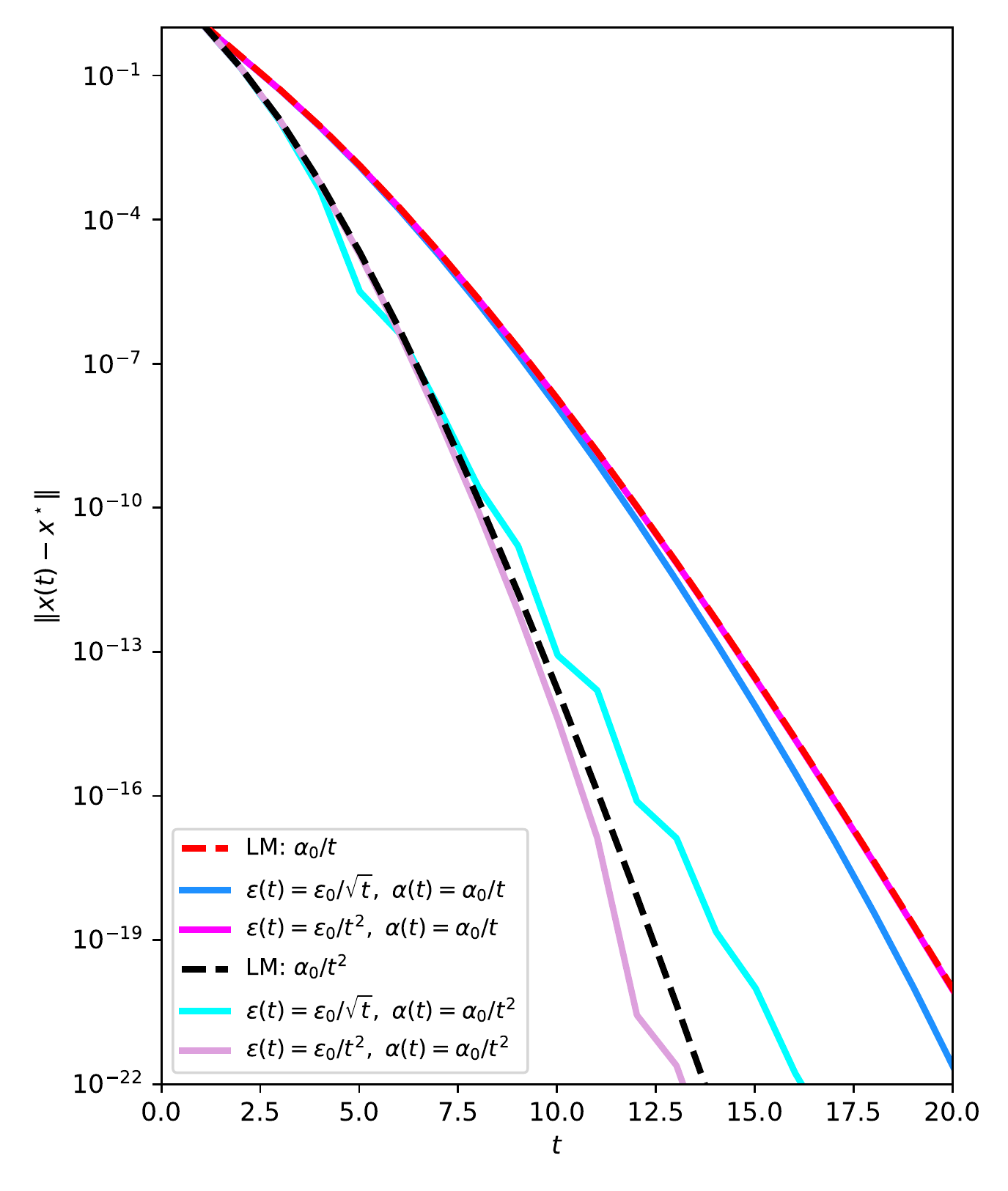}
		\caption{\label{fig::quadspeed2} Same experiments as in Figure~\ref{fig::quadspeed} but in the setting $\gamma=\beta=1$ (so that Newton's method converges in one iteration), for large values of $\eps_0$ (left figures) and small $\eps_0$ (right). The use of a large step-size $\gamma$ makes this setting out of the scope of our theoretical results. Due to fast convergence, different behaviors can be observed for large values of $\eps_0$ and smaller ones.}
	\end{figure}
	We present two set of experiments that illustrate our main results from Sections~\ref{sec::Control} and~\ref{sec::Quads}. We first detail the general methodology.
	\subsection{Methodology}
	We compare the solutions of \eqref{eq::CN}, \eqref{eq::LM} and \ref{eq::VMDINAVD} obtained for strongly convex functions in dimension $n=100$. Since closed-form solutions are not available, they are estimated via discretization schemes with small\footnote{In practice Newton's method would converge faster when used with $\gamma=\beta=1$, here we take a smaller step-size $\gamma$ to more accurately approximate the solutions of the ODEs.} step-sizes $\gamma=10^{-1}$.
	All ODEs are initialized at $x_0\in\R^n$, where each coordinate of $x_0$ belongs to $\{-1,1\}$. For all functions this means that $x_0$ is approximately at distance $\sqrt{n}$ of $x^\star$, hence not very close to $x^\star$. We then approximated the ODEs on time intervals that are long-enough to observe both non-asymptotic and asymptotic convergence behaviors.
	We used Euler semi-explicit schemes obtained by solving linear systems, for the sake of stability.
	The resulting algorithms are detailed in Appendix~\ref{supp::Exp}.
	\subsection{First Experiment: Distance between Trajectories}
	We begin with an empirical validation of the results of Section~\ref{sec::Control} on the distance between $x$, $x_{LM}$ and $x_N$.
	Each of Figures~\ref{fig::expmt}, \ref{fig::logsumexpexp} and~\ref{fig::LP}
	corresponds to a different strongly convex function, specified below its corresponding figure. To ensure strong convexity, each function contains a quadratic term of the form $\Vert Ax\Vert^2$, where $A$ is symmetric positive definite.

	Several observations can be made from the numerical results, but we first note on the right plots of Figures~\ref{fig::expmt} to~\ref{fig::LP} that $x_N$ always converges asymptotically linearly (i.e., exponentially fast). This is also the case for $x$ and $x_{LM}$ in some (but not all) cases. This is important because $\norm{x(t)-x_N(t)}\leq \norm{x(t)-x^\star} + \norm{x_N(t)-x^\star}$, so if both $x$ and $x_N$ converge linearly, then the bounds of Theorems~\ref{thm::mainEpsLarge} and~\ref{thm::mainGenResult} need not be asymptotically tight. That being said, the strength of these results is to be non-asymptotic and this is highlighted by the experiments as we now explain.

	The left and middle plots of Figures~\ref{fig::expmt}, \ref{fig::logsumexpexp} and~\ref{fig::LP} are consistent with Theorems~\ref{thm::mainEpsLarge} and~\ref{thm::mainGenResult},  since the distances $\norm{x(t)-x_{N}(t)}$ and $\norm{x(t)-x_{LM}(t)}$ decrease relatively fast to zero. Again, when $x$ converges rapidly to $x^\star$ this is not very insightful, however, the main interest of our theorems appears on the left of Figures~\ref{fig::logsumexpexp} and~\ref{fig::LP}: the blue and green curves, corresponding to slowly decaying choices of $\eps$, converge more slowly than other trajectories. However, when taking faster decays, we recover fast convergence and closeness to $x_N$ (this is particularly true for the purple curve).
	Very similar observations are made w.r.t.\ $x_{LM}$ on the middle plots. Despite not being stated in the theorems of Section~\ref{sec::Control}, the experiments match the intuition that when $\eps> \alpha$, $x$ may be closer to $x_N$ and when $\eps\leq\alpha$, $x$ would rather be closer to $x_{LM}$. This is more noticeable on the top rows of the figures, where $\alpha$ is not integrable.

	Figure~\ref{fig::LP} suggests that the bounds in Theorem~\ref{thm::mainGenResult} are rather tight for small $t$, since, for example, the blue and green curves on the left show a relatively slow vanishing of $\norm{x(t)-x_N(t)}$ for slowly decaying $\eps$. The bounds seem however often too pessimistic for large $t$, for which the second part of our study provides better insights (see Section~\ref{sec::Quads} and below).
	Interestingly, slow decays of $\eps$ might result in faster convergence for $x$ than fast decays (and also faster convergence than $x_{LM}$), notably on Figure~\ref{fig::expmt}. We also note that $\eps(t)=\eps_0/t$ combined either with $\alpha(t)=\alpha_0/t$ or $\alpha(t)=\alpha_0/t^2$ seems to very often yield fast convergence in these experiments.

	\subsection{Second Experiment: Empirical Validation of Theorem~\ref{thm::MainResAsymptotic}\label{sec::secondexp}}
	We now turn our attention to the solutions $x$, $x_N$ and $x_{LM}$ for a quadratic function of the form $f(y)=\frac{1}{2}\Vert Ay\Vert^2$, $y\in\R^n$, and for several choices of $\eps$ and $\alpha$. The results in Figure~\ref{fig::quadspeed} exactly match  the expected behavior summarized in Table~\ref{tabl::speedsummary}. Indeed, looking first at the right-hand side of Figure~\ref{fig::quadspeed}, $x$ is as fast as the corresponding\footnote{That is, the solution of \eqref{eq::LM} for the same $\alpha$ as that considered for \eqref{eq::VMDINAVD}.} $x_{LM}$ when $\eps$ is not integrable and regardless of $\alpha$, and $x$ is faster when $\eps$ is non-integrable. Then on the left-hand side, when comparing to $x_N$, $x$ is slower in settings where $\alpha$ is larger than $\eps$ and non-integrable (red curves), or almost as fast when $\alpha$ is integrable (pink curve). However, acceleration w.r.t.\ to $x_N$ is indeed achieved for non-integrable $\eps$ regardless of $\alpha$ (first-two blue curves), and the rate is the same as that of $x_N$ when $\eps$ is integrable (third blue curve). Finally, on Figure ~\ref{fig::quadspeed2} we empirically investigate the case of large step-sizes $\gamma$ and take $\gamma=\beta=1$ (optimal for Newton's method on quadratics). Due to large step-sizes our theoretical results do not hold and the choice of $\eps_0$ matters more. This evidences a trade-off between $\eps_0$ and $\gamma$ in the case of discrete algorithms which is not present in the analysis of ODEs and is worth investigating in a future work.

	\section{Conclusions and Perspectives}
	We introduced a general ODE \ref{eq::VMDINAVD} featuring variable mass, and provided a deep understanding on the behavior of its solutions w.r.t.\ time dependent control parameters $\eps$ and $\alpha$, both, asymptotically and non-asymptotically. We can conclude that \eqref{eq::VMDINAVD} is indeed of  (regularized) Newton type, since it can be controlled to be  close to both \eqref{eq::CN} and \eqref{eq::LM}. Yet we also showed that \eqref{eq::VMDINAVD} fundamentally differs from the other two dynamics in its nature. In particular, Theorem~\ref{thm::MainResAsymptotic} and the numerical experiments emphasized that $\eps$ and $\alpha$ can accelerate (or slow down) \eqref{eq::VMDINAVD} w.r.t.\ \eqref{eq::CN} and \eqref{eq::LM}. We also note that our bounds in Theorems~\ref{thm::mainEpsLarge} and~\ref{thm::mainGenResult} seem relatively tight, in particular for functions with large gradients (see Figure~\ref{fig::LP}).
	Our contribution yields a complete and satisfying picture on the relation between the three systems, which was only partially understood. We believe that our results build a strong foundation for the development of algorithms that combine the best properties of first- and second-order optimization methods.

	As for future work, we showed that \eqref{eq::VMDINAVD} is promising from an optimization perspective. So far we approximated solutions of \eqref{eq::VMDINAVD} via schemes that required solving linear systems (this is also true for \eqref{eq::CN} and \eqref{eq::LM}). Our new understanding on $(\eps,\alpha)$ paves the way towards designing new Newton-like algorithms with a significantly reduced computational cost, which is crucial for large-scale optimization. Another open question is whether it is possible to preserve the properties evidenced in this work when $\eps$ is defined in a closed-loop manner (formally depending on $x$ rather than on $t$). Finally, it would be worth investigating how the current work can be extended to general convex and/or non-smooth functions.

	\FloatBarrier

	\section*{Acknowledgment}
	C. Castera, J. Fadili and P. Ochs are supported by the ANR-DFG joint project TRINOM-DS under numbers  ANR-20-CE92-0037-01 and OC150/5-1. The numerical experiments were made thanks to the development teams of the following libraries: Python \citep{rossum1995python}, Numpy \citep{walt2011numpy} and
	Matplotlib \citep{hunter2007matplotlib}.

	\newpage
	\appendix
	\begin{flushleft}
		{\LARGE \textbf{Appendices}}
	\end{flushleft}


	\section{Equivalent First-order System and Existence of Solutions}\label{app:FOS}
	\subsection{First-order Equivalent Formulation}
	We reformulate \eqref{eq::VMDINAVD} as a system of ODE involving only first-order time derivatives and the gradient of $f$. For this purpose, notice that for all $t>0$ \eqref{eq::VMDINAVD} can be rewritten as,
	\begin{align}\label{eq::intVMDINAVD}
		\begin{split}
			\frac{\diff}{\diff t} \left[\eps(t)\dx(t)\right] + \beta \frac{\diff}{\diff t}\nabla f(x(t)) + \alpha(t)\dx(t) - \eps'(t)\dx(t)+ \nabla f(x(t)) = 0, \quad t\geq0.
		\end{split}
	\end{align}
	We then integrate \eqref{eq::intVMDINAVD} for all $t\geq 0$,
	\begin{equation}\label{eq::intVMDINAVD2}
		\eps(t)\dx(t) + \beta \gf(x(t)) - \eps_0\dx_0 - \beta \gf(x_0) + \int_{0}^t (\alpha(s)-\eps'(s))\dx(s)+ \nabla f(x(s)) \diff s= 0.
	\end{equation}
	For all $t\geq 0$, we define the variable,
	$$z(t)=\int_{0}^t (\alpha(s)-\eps'(s))\dx(s)+ \nabla f(x(s)) \diff s - \eps_0\dx_0 - \beta \gf(x_0).$$
	We differentiate $z$, for all $t>0$,
	$\dot z(t)=(\alpha(t)-\eps'(t))\dx(t)+ \nabla f(x(t))$, so that we can rewrite \eqref{eq::intVMDINAVD2} as,
	\begin{equation*}
		\begin{cases}
			\eps(t)\dx(t) + \beta \gf(x(t)) + z(t)= 0
			\\\dot z(t) - (\alpha(t)-\eps'(t))\dx(t) - \nabla f(x(t)) = 0
		\end{cases}, \quad t\geq0.
	\end{equation*}
	We substitute the first line in the second-one,
	\begin{equation}\label{eq::intDINAVD4}
		\begin{cases}
			\eps(t)\dx(t) + \beta \gf(x(t)) + z(t)= 0
			\\\beta\dot z(t) - \beta(\alpha(t)-\eps'(t) - \frac 1 \beta \eps(t))\dx(t) + z(t) = 0
		\end{cases}, \quad t\geq0.
	\end{equation}
	To ease the readability, we recall the notation $\nu(t) = \alpha(t)-\eps'(t) - \frac 1 \beta \eps(t)$ from Section~\ref{sec::prelim}. Then define for all $t\geq 0$, $y(t)= z(t)- \nu(t) x(t)$, and differentiate,
	$\dot y(t)=\dot z(t)-\nu(t)\dx(t) - \nu'(t)x(t)$.
	We finally rewrite \eqref{eq::intDINAVD4} as,
	\begin{equation*}
		\begin{cases}
			\eps(t)\dx(t) + \beta\gf(x(t)) &+ \nu(t)x(t) + y(t)= 0
			\\\dot y(t)+ \nu'(t)x(t) &+ \frac{\nu(t)}{\beta} x(t) + \frac 1 \beta y(t) = 0
		\end{cases}.{}
	\end{equation*}
	which is \eqref{eq::gVMDINAVD}. Finally, the initial condition on $y$ is $$y(0) = z(0)-\nu(0)x(0) = -\eps_0\dx_0 - \beta\gf(x_0) - (\alpha_0-\eps'_0-\frac{1}{\beta}\eps_0)x_0.$$

	\begin{remark}
		Notice that the quantity $\nu(t)= \alpha(t)-\eps'(t) - \frac 1 \beta \eps(t)$ involved in \eqref{eq::gVMDINAVD} also plays a key role in our analysis of Section~\ref{sec::Control}, see e.g., \eqref{eq::boundInt2}. In particular the sign of $\nu(t)$ changes the nature of \eqref{eq::VMDINAVD} and is related to Assumption~\ref{ass:largeepsilon}.
	\end{remark}

	\subsection{Local Solutions are Global}\label{app::localisglobal}
	Using the formulation \eqref{eq::gVMDINAVD}, we proved local existence and uniqueness of solutions of \eqref{eq::VMDINAVD} in Section~\ref{sec::prelim}. Using the same notations, we justify that the local solution $(x,y)$ actually exists globally.
	According to Lemma~\ref{lem::LyapLargeEps}, the Lyapunov function $U(t) = \frac{\eps(t)}{2} \norm{\dot{x}(t)}^2 + f(x(t))-f(x^\star)$ is non-negative and decreasing. Thus, it is uniformly bounded on $\R_+$ and the same holds for $t\mapsto f(x(t))$ since for all $t\geq 0$, $U(t)\geq f(x(t))$. Then, $f$ is coercive by assumption, so $x$ is uniformly bounded on $\R_+$ (otherwise $f(x)$ could not remain bounded).
	We now prove that $y$ is also uniformly bounded. From \eqref{eq::gVMDINAVD}, for all $t>0$, $\dot y(t) = -\frac 1 \beta y(t) - (\frac{\nu(t)}{\beta} + \nu'(t))x(t)$ so we can use the following integrating factor,
	$$y(t) = e^{-\frac{t}{\beta}}y_0 - e^{-\frac{t}{\beta}}\int_0^t \frac 1 \beta e^{\frac{s}{\beta}}(\nu(s) + \beta\nu'(s))x(s)\diff s.$$
	Using triangle inequalities, for all $t\geq 0$,
	\begin{multline}\label{eq::boundy}
		\Vert y(t)\Vert \leq e^{-\frac{t}{\beta}}\Vert y_0\Vert + \sup_{s\geq 0}  \Vert (\nu(s) + \beta\nu'(s)) x(s)\Vert e^{-\frac{t}{\beta}} \int_0^t \frac 1 \beta e^{\frac{s}{\beta}}\diff s
		\\
		\leq \Vert y_0\Vert + \sup_{s\geq 0}  \Vert (\nu(s) + \beta\nu'(s)) x(s)\Vert.
	\end{multline}
	Using the definition of $\eps$ and $\alpha$ from Sections~\ref{sec::pbstatement} and~\ref{sec::prelim}, observe that $\eps$, $\alpha$, $\eps'$ and $\alpha'$ are all bounded on $\R_+$, and $\eps''$ is assumed to be bounded. So $\nu$ and $\nu'$ are bounded, and since we also proved that $x$ is uniformly bounded on $\R_+$, we deduce from \eqref{eq::boundy} that $y$ is uniformly bounded as well. Hence, the unique local solution $(x,y)$ is global.

		\subsection{Existence and Uniqueness for Non-smooth Functions}\label{app:nonsmooth}
		Following Remark~\ref{rem::nonsmooth}, we here provide the main arguments to extend \eqref{eq::VMDINAVD} to the non-smooth setting. Assume, in this section only, that $f$ is convex, proper, lower semi-continuous, and (without loss of generality) that $\mathrm{dom}(f)=\R^n$. Consider
		\begin{equation}\label{eq::nonsmoothVM}
			\begin{cases}
				\eps(t)\dx(t) + \beta\partial f(x(t)) + \nu(t)x(t) + y(t)&\ni 0,
				\\\dot y(t)+ \nu'(t)x(t) + \frac{\nu(t)}{\beta} x(t) + \frac 1 \beta y(t) &= 0,
			\end{cases} \quad\text{ for a.e. } t\geq 0,
		\end{equation}
		where $\partial f$ is the sub-differential of $f$.
		Note that when $f$ is differentiable, \eqref{eq::nonsmoothVM} boils down to \eqref{eq::gVMDINAVD}.
		An absolutely continuous mapping $(x,y):\R_+\to\R^n\times\R^n$ that satisfies \eqref{eq::nonsmoothVM} is called solution.
		Remark that defining $\tilde G(t,x(t),y(t)) = (\frac{\nu(t)x(t)+y(t)}{\eps(t)},\nu'(t)x(t)+\frac{\nu(t)}{\beta} x(t) + \frac{1}{\beta}y(t) )$ and  $F(t,x(t),y(t)) = (\frac{\beta}{\eps(t)} f(x(t)),0)$, \eqref{eq::nonsmoothVM} rewrites $(\dx(t),\dot y(t)) + \partial F(t,x(t),y(t)) + \tilde G(t,x(t),y(t))\ni 0$. Then under our assumptions on $\eps$ and $\alpha$, $\tilde G(t,\cdot,\cdot)$ is Lipschitz continuous which allows showing existence of a solution via the theory of non-linear semigroups; see \cite[Proposition~3.12]{brezis1973ope}. Note that the chain rule on $(x,y)$ holds for a.e. $t\geq 0$ thanks to absolute continuity, which allows extending our Lyapunov analysis as well.

\section{Generalization to Strictly Convex Functions}\label{app::strictconvex}
		Following Remark~\ref{rem::extensions}, we detail how to generalize Section~\ref{sec::Control} beyond strongly convex functions.
		Assume, here only, that $f$ is strictly convex. Since $f$ is continuous and the trajectories are contained in some compact set $\mathsf{K}_0$ (see the proof of Theorem~\ref{thm::mainEpsLarge}), it follows that $f$ is uniformly convex on $\mathsf{K}_0$; see \cite[Proposition~10.15]{bauschke2011convex}. This implies that $\nabla f$ is uniformly monotone \citep[Example~22.3]{bauschke2011convex} on $\mathsf{K_0}$. That is, there exist an increasing function $\phi:\R_+\to \R_+$ such that $\phi(0)=0$, and $\forall y_1,y_2\in\mathsf{K_0}$, $\langle \nabla f(y_1)-\nabla f(y_2), y_1-y_2\rangle\geq \phi(\norm{y_1-y_2})$. Note that when $f$ is strongly convex, $\phi(s)=\mu_\mathsf{K_0} s^2$, $\forall s\geq 0$.
		We deduce that
		$
		\frac{\phi(\norm{y_1-y_2})}{\norm{y_1-y_2}}  \leq  \norm{\nabla f(y_1)-\nabla f(y_2)}
		$, which allows extending \eqref{eq::boundwithgrad} and our analysis to strictly convex $f$ whose associated $\phi$ is such that $\phi(s)/s\xrightarrow[s\to 0]{}0$. In that case, $\phi$ is called a growth function, and this covers many applications; see \cite{drusvyatskiy2021nonsmooth,ochs2019}.


	\section{Proof of Theorem~\ref{thm::mainGenResult}: Key arguments}\label{supp::GenThm}
	This section is devoted to proving the general result of Section~\ref{sec::Control}. Fix some constants $c_1,c_2>0$ and let $\eps$ and $\alpha$ such that Assumption~\ref{ass:Onlysubexp} is satisfied with these constants.
	Let $x$ be the corresponding solution of \eqref{eq::VMDINAVD}, $x_N$, and $\xlm$ that of \eqref{eq::CN} and \eqref{eq::LM}, respectively.
	Following the same arguments as in the beginning of the proof of Theorem~\ref{thm::mainEpsLarge}, for all $t\geq 0$, $x(t)$, $x_N(t)$ and $x_{LM}(t)$ belong to the bounded set $\mathsf{K}_0$ defined in that proof.
	Since $f$ is $\mu$-strongly convex on $\mathsf{K}_0$, the proof relies again on bounding differences of gradients like in \eqref{eq::boundwithgrad}.
	Following the exact same steps as in the proof of Theorem~\ref{thm::mainEpsLarge}, we know the closed form of $\gf(x_N)$ given in \eqref{eq::gradN}, and an expression for $\gf(x)$ given in \eqref{eq::gradDIN}.
	For all $s\geq 0$, we have the identity,
	\begin{equation}\label{eq::alphatrick}
		e^{ \frac s \beta}\dx(s)  =   e^{ \frac s \beta} \dx(s)  + \frac{1}{\beta} e^{ \frac s \beta}x(s) - \frac{1}{\beta} e^{ \frac s \beta}x(s) = \frac{\diff}{\diff s}(e^{ \frac s \beta}x(s)) - \frac{1}{\beta} e^{ \frac s \beta }x(s),
	\end{equation}
	which we use to perform an integration by parts
	\begin{equation*}
		\int_{0}^{t} e^{ \frac s \beta} \alpha(s) \dx(s)\diff s = \left[\alpha(s)e^{ \frac s \beta }x(s)\right]_0^t -\int_{0}^{t} \left(\alpha'(s)+\frac{\alpha(s)}{\beta}\right) e^{ \frac s \beta }x(s)\diff s.
	\end{equation*}
	Therefore,
	\begin{equation}\label{eq::alphatrick3}
		e^{ -\frac t \beta}\int_{0}^{t} e^{ \frac s \beta } \alpha(s) \dx(s)\diff s = \alpha(t)x(t) - e^{ -\frac t \beta}\alpha_0 x_0 -\int_{0}^{t}e^\frac{s-t}{\beta} \left(\alpha'(s)+\frac{\alpha(s)}{\beta}\right) x(s)\diff s,
	\end{equation}
	and we can substitute in \eqref{eq::gradDIN},
	\begin{multline}\label{eq::gradDIN3}
		\beta\gf(x(t))=\beta e^{-\frac t \beta}\gf(x_0) + e^{-\frac t \beta }\eps_0\dx_0 -\eps(t)\dx(t)  + \int_{0}^{t} e^\frac{s-t}{\beta} \left(\frac 1 \beta \eps(s) + \eps'(s)\right)\dx(s)\diff s
		\\
		- \alpha(t)x(t) + e^{ -\frac t \beta}\alpha_0 x_0  + \int_{0}^{t} e^\frac{s-t}{\beta}\left(\alpha'(s)+\frac{\alpha(s)}{\beta}\right)x(s)\diff s.
	\end{multline}
	One then proves \eqref{eq::BoundXN} by inserting \eqref{eq::gradDIN3} in \eqref{eq::boundwithgrad} and by using the uniform boundedness of $x$ and Assumption~\ref{ass:Onlysubexp}. The proof of \eqref{eq::BoundXLM} follows the same arguments based on the fact that $\forall t\geq 0$,
	\begin{equation*}
		\gf(\xlm(t))=e^{-\frac t \beta}\gf(x_0) - e^{-\frac t \beta}\int_{0}^{t} \frac{1}{\beta}e^{ \frac s \beta} \alpha(s) \dxlm(s)\diff s.
	\end{equation*}


	\section{Integrability of \texorpdfstring{$\varphi$}{φ} and Additional Asymptotic Computations}\label{app::integrability}
	Below we prove Lemma~\ref{lem::phiintegrable}.
	\begin{proof}
		We suppose that Assumptions~\ref{ass:smooth} and~\ref{ass:subexp} hold. As stated in Remark~\ref{rem::olver}, since $\varphi$ is continuous, we only need to check its integrability when $t$ tends to $+\infty$. Let $t>0$, we first establish some useful identities, we omit the dependence on $t$ for the sake of readability.
		\begin{align*}
			p' = \frac{\alpha'\eps-(\alpha+\beta\lambda)\eps'}{\eps^2},
			\quad \text{and}\quad
			p'' = \frac{\alpha''\eps^2-2\alpha'\eps'\eps-(\alpha+\beta\lambda)\eps''\eps + 2(\alpha+\beta\lambda)(\eps')^2}{\eps^3}.
		\end{align*}
		Then,
		\begin{align}\label{eq::factorR}
			\begin{split}
				r  &= \frac{p^2}{4}\left(1+ \frac{2p'}{p^2} - \frac{4\lambda}{\eps p^2}\right) = \frac{(\alpha+\beta\lambda)^2}{4\eps^2}\left(1+\frac{2p' \eps ^2}{(\alpha +\beta\lambda)^2} - \frac{4\lambda\eps }{(\alpha +\beta\lambda)^2}\right)
				\\
				&=\frac{(\alpha +\beta\lambda)^2}{4\eps ^2}\left(1+\frac{2\alpha' \eps }{(\alpha +\beta\lambda)^2} - \frac{2\eps' }{(\alpha +\beta\lambda)} - \frac{4\lambda\eps }{(\alpha +\beta\lambda)^2}\right).
			\end{split}
		\end{align}
		An important consequence of Assumption~\ref{ass:subexp} is that $\vert\eps'(t)\vert = o(\eps(t))$, $\vert\eps''(t)\vert = o(\eps'(t))$ (and the same holds for $\alpha$ w.r.t.\ to its derivatives). Therefore, we deduce from \eqref{eq::factorR} that
		\begin{equation*}
			r(t)\sim_{+\infty}  \frac{(\alpha(t)+\beta\lambda)^2}{4\eps(t)^2},
		\end{equation*}
		and we note that $1/r$ decays at the same speed as $\eps^2$, which will be useful later. In order to study $\varphi$, we now differentiate $r$,
		\begin{align}\label{eq::rprime}
			\begin{split}
				r' &= \frac{p'p}{2}\left(1+ \frac{2p'}{p^2} - \frac{4\lambda}{\eps p^2}\right) + \frac{1}{4}\left( 2p'' - \frac{4(p')^2}{p} + \frac{8\lambda p'}{\eps p}+ \frac{4\lambda\eps'}{\eps^2} \right)
				\\&= \frac{2p'}{p}r + \frac{1}{4}\left( 2p'' - \frac{4(p')^2}{p} +\frac{8\lambda p'}{\eps p} + \frac{4\lambda\eps'}{\eps^2}  \right),\quad \text{and,}
			\end{split}
		\end{align}

		\begin{multline}\label{eq::rsec}
			r'' = 2\frac{p''p-(p')^2}{p^2}r  + \frac{2p'}{p}r'
			\\
			+ \frac{1}{4}\left( 2p''' + 4\frac{(p')^3-2p''p'p}{p^2}+8 \lambda\frac{p''p\eps - (p')^2\eps - p'p\eps'}{\eps^2 p^2} + \frac{4\lambda\eps''}{\eps^2} - \frac{8\lambda(\eps')^2}{\eps^3}  \right).
		\end{multline}
		Then, to justify that $\varphi$ is integrable, we prove that $\frac{r''}{r^{3/2}}$ and $\frac{(r')^2}{r^{5/2}}$ are integrable. Since we know that $1/r$ decays at the same speed as $\eps^2$, we can equivalently show that $\eps^3r''$ and $\eps^5(r')^2$ are integrable.
		To this aim we fully expand all the terms in \eqref{eq::rprime} and \eqref{eq::rsec}.
		We first deal with \eqref{eq::rprime}, it holds that:
		\begin{align*}
			r'^2&\eps^{5} =
			\left[- \frac{ (\alpha+\beta\lambda)^{2} \left(- \frac{4 \lambda \eps }{ (\alpha+\beta\lambda)^{2}} + 1 + \frac{\left(- {2  (\alpha+\beta\lambda) \eps' } + {2 \alpha'\eps }\right) }{ (\alpha+\beta\lambda)^{2}}\right) \eps' }{2 \sqrt{\eps}}\right.
			\\ &\left.
			+ \frac{ (\alpha+\beta\lambda)^{2}\sqrt{\eps}}{4} \left(- \frac{4 \lambda \eps' }{ (\alpha+\beta\lambda)^{2}} + \frac{8 \lambda \alpha'  \eps }{ (\alpha+\beta\lambda)^{3}} + \frac{2 \left(- \frac{2  (\alpha+\beta\lambda) \eps' }{\eps } + {2 \alpha'\eps }\right) \eps'   }{ (\alpha+\beta\lambda)^{2}}
			\right.\right.
			\\&\left.\left.
			+ \frac{\left(\frac{4  (\alpha+\beta\lambda) \eps'^{2} }{\eps } - {2  (\alpha+\beta\lambda) \eps'' } - {4 \alpha'  \eps' } + {2 \alpha''\eps }\right) }{ (\alpha+\beta\lambda)^{2}} - \frac{2 \left(-{2  (\alpha+\beta\lambda) \eps' } + {2 \alpha'\eps }\right) \alpha'   }{ (\alpha+\beta\lambda)^{3}}\right)
			\right.\\
			&\left. + \frac{ (\alpha+\beta\lambda)}{2} { \left(- \frac{4 \lambda \eps }{ (\alpha+\beta\lambda)^{2}} + 1 + \frac{\left(- {2  (\alpha+\beta\lambda) \eps' } + {2 \alpha'\eps }\right) }{ (\alpha+\beta\lambda)^{2}}\right)} \alpha' \sqrt{\eps}
			\phantom{\frac{\frac{\beta}{\beta}}{\beta}}
			\hspace{-1.7cm}\phantom{\frac{\left( \frac{\frac{a'}{a'}}{a'}\right)}{a^2}{\sqrt{.}}
			}\right]^{2}.
		\end{align*}
		The computations for $\eps^2r''$ are omitted since they are longer but very similar.

		We analyze the integrability of each of the terms above (and those of $\eps^2r''$). By Assumption~\ref{ass:subexp}, $\eps'$, $\eps''$ and $\eps'''$ are integrable, as well as $\alpha'$, $\alpha''$ and $\alpha'''$, this justifies the integrability of most of the terms. We also need $\frac{(\eps')^2}{\eps}$ and $\frac{(\eps')^3}{\eps}$ to be integrable, which holds by Assumption~\ref{ass:subexp}. Overall, $\varphi$ is integrable on $\R_{+}$.
	\end{proof}

	We now state and prove the a result used at the end of the proof of Theorem~\ref{thm::MainResLG}.
	\begin{lemma}\label{lem::remainderterm}
		Under Assumptions~\ref{ass:smooth} and~\ref{ass:subexp}, for all $s\geq 0$, $$\frac{1}{16}\left(\frac{2p'(s)}{p(s)^{3/2}} - \frac{4\lambda\sqrt{\eps(s)}}{(\alpha(s)+\beta\lambda)^{3/2}}\right)^2 = \frac{\lambda^2\eps(s)}{(\alpha(s)+\beta\lambda)^{3}}+o(\eps(s)).$$
	\end{lemma}
	\begin{proof}
		We omit the time dependence on $s\geq 0$ for the sake of readability. Using Assumption~\ref{ass:smooth} we can define and expand the following quantity,
		\begin{align*}
			&\frac{1}{16}\left(\frac{2p'}{p^{3/2}} - \frac{4\lambda\sqrt{\eps}}{(\alpha+\beta\lambda)^{3/2}}\right)^2
			=\frac{\lambda^2\eps}{(\alpha+\beta\lambda)^{3}}  - \frac{p'\lambda\sqrt{\eps}}{p^{3/2}(\alpha+\beta\lambda)^{3/2}} + \frac{(p')^2}{4p^{3}}
			\\&= \frac{\lambda^2\eps}{(\alpha+\beta\lambda)^{3}}
				- \frac{\lambda(\alpha'\eps - (\alpha+\beta\lambda)\eps')}{(\alpha+\beta\lambda)^{3}}
				+ \frac{(\alpha')^2\eps + \frac{(\eps')^2}{\eps}(\alpha+\beta\lambda)^2 - 2 \alpha'\eps'(\alpha+\beta\lambda)}{4(\alpha+\beta\lambda)^{3}}.
		\end{align*}
		Assumption~\ref{ass:subexp}, implies in particular that $\vert\eps'(t)\vert=o(\eps(t))$ and that $\alpha'(t)\to 0$, which we use in the equality above to obtain the desired conclusion.
	\end{proof}


	\section{Additional Experimental Details}\label{supp::Exp}
	We used Euler discretization schemes with fixed step-size $\gamma>0$, for approximating the solutions of the three ODEs considered in Section~\ref{sec::exp}. For a trajectory $x$, at times $t_k=\gamma k$, $k\in\N$, we denote $x(t_k)\stackrel{\textrm{def}}{=} x^{(k)}$.
	We approximated \eqref{eq::CN} by explicit discretization, $\forall k\in\N$,
	\begin{equation}\label{eq::discNewton}
		x_N^{(k+1)} = x_N^{(k)} - \gamma \left[\beta\Hf(x_N^{(k)})\right]^{-1}\gf(x_N^{(k)}).
	\end{equation}
	Then, defining $\eps_k = \eps(t_k)$ and $\alpha_k = \alpha(t_k)$, \eqref{eq::LM} and \eqref{eq::VMDINAVD} are obtained via Euler semi-implicit discretization. The solution of \eqref{eq::LM} is approximated by,
	\begin{equation}\label{eq::discLM}
		x_{LM}^{(k+1)} = x_{LM}^{(k)} - \gamma \left[\alpha_k I_n+\beta\Hf(x_{LM}^{(k)})\right]^{-1}\gf(x_{LM}^{(k)}),
	\end{equation}
	where $I_n$ is the identity matrix on $\R^n$. The solution of \eqref{eq::VMDINAVD} is similarly,
	\begin{equation}\label{eq::discVM}
		x^{(k+1)} = x^{(k)} +  \left[(\eps_k+\gamma\alpha_k) I_n+\gamma\beta\Hf(x^{(k)})\right]^{-1}\left(\eps_k(x^{(k)}-x^{(k-1)})-\gamma^2\gf(x^{(k)})\right).
	\end{equation}
	As sanity check, one can see that for $\eps_k=0$, \eqref{eq::discVM} is equivalent to \eqref{eq::discLM}, which is itself equivalent to \eqref{eq::discNewton} when $\alpha_k=0$.

	\FloatBarrier
	\bibliographystyle{plainnat}
	\bibliography{biblio.bib}

\end{document}